\documentclass[14pt]{extarticle}
\usepackage{amssymb,amsfonts,amsthm,amsmath,bbold}
\usepackage{mathrsfs,pifont}
\usepackage{mathabx}
\usepackage[all]{xy}
\usepackage{array}

\usepackage[top=1.2in, bottom=1.2in, left=0.5in,
right=0.5in]{geometry}

\usepackage[font={small,sl}]{caption}
\usepackage{tikz}
\usepackage{tikz-cd}
\usetikzlibrary{matrix,arrows}

\newcommand{\Rd}{{\mathsf{R}}^{\raisebox{0.5mm}{$\scriptscriptstyle \bullet$}}}

\usepackage{hyperref}
\usepackage[backrefs,msc-links]{amsrefs}

\newcommand{\C}{\mathbb{C}}

\newcommand{\Ct}{\mathbb{C}^\times}

\newcommand{\A}{\mathbb{A}}
\newcommand{\Dd}{\mathbb{D}}
\newcommand{\Z}{\mathbb{Z}}
\newcommand{\R}{\mathbb{R}}

\newcommand{\bT}{\mathsf{T}}
\newcommand{\bA}{\mathsf{A}}
\newcommand{\bff}{\mathsf{f}}
\newcommand{\bG}{\mathsf{G}}
\newcommand{\bg}{\mathsf{g}}
\newcommand{\bGt}{{\widetilde{\mathsf{G}}}}
\newcommand{\bGa}{{\mathsf{G}_{\textup{Aut}}}}
\newcommand{\bGas}{{\mathsf{G}_{\textup{Aut},\omega}}}
\newcommand{\sP}{\mathsf{P}}

\newcommand{\bY}{\mathsf{Y}}

\newcommand{\bZ}{\mathsf{Z}}
\newcommand{\bZb}{\overline{\mathsf{Z}}}
\newcommand{\bL}{\mathsf{L}}

\newcommand{\bM}{\mathsf{M}}
\newcommand{\bR}{\mathsf{R}}

\newcommand{\Gr}{\mathsf{Gr}}
\newcommand{\bB}{\mathsf{B}}

\newcommand{\bX}{\mathsf{X}}
\newcommand{\bj}{\mathsf{j}}

\newcommand{\bP}{\mathbb{P}}
\newcommand{\cK}{\mathscr{K}}

\newcommand{\cL}{\mathscr{L}}
\newcommand{\Lamp}{{\cL_\textup{amp}}}
\newcommand{\cR}{\mathscr{R}}
\newcommand{\cP}{\mathscr{P}}
\newcommand{\cU}{\mathscr{U}}

\newcommand{\cM}{\mathscr{M}}

\newcommand{\cV}{\mathscr{V}}
\newcommand{\cW}{\mathscr{W}}
\newcommand{\cE}{\mathscr{E}}

\newcommand{\Ths}{T^{1/2}\left[\bX/\bG\right]}

\newcommand{\symr}{{\scriptscriptstyle /\!/\!/\!/}}

\newcommand{\tS}{\widetilde{S}}
\newcommand{\phib}{\phi^\circ}

\newcommand{\bmu}{\boldsymbol{\mu}}

\newcommand{\bu}{\mathbf{u}}

\newcommand{\ust}{{\textup{ust}}}
\newcommand{\sst}{{\textup{sst}}}

\newcommand{\bGamma}{\boldsymbol{\Gamma}}
\newcommand{\bPhi}{\boldsymbol{\Phi}}
\newcommand{\tgam}{\overline{\gamma}}
\newcommand{\txx}{\widetilde{x}}

\newcommand{\cS}{\mathscr{S}}

\newcommand{\fg}{\mathfrak{g}}
\newcommand{\fgb}{\boldsymbol{\mathfrak{g}}}
\newcommand{\fp}{\mathfrak{p}}

\newcommand{\rd}{/\!\!/\!\!/\!\!/}
\newcommand{\rdd}{/\!\!/}

\newcommand{\aroof}{\widehat{a}} 

\newcommand{\vth}{\vartheta}
\newcommand{\tf}{\tilde f} 

\newcommand{\cO}{\mathscr{O}}
\newcommand{\Hd}{{H}^{\raisebox{0.5mm}{$\scriptscriptstyle \bullet$}}}

\newcommand{\tO}{\widehat{\mathscr{O}}}
\newcommand{\vir}{\textup{vir}}
\newcommand{\mero}{{\textup{mero}}}
\newcommand{\tp}{\textup{top}}

\newcommand{\Stab}{\mathsf{Stab}}
\newcommand{\Attr}{\mathsf{Attr}}
\newcommand{\Db}{D^{\textup{b}}}

\newcommand{\tX}{\widetilde{\bX}}
\newcommand{\comp}{\textup{compact}}

\DeclareMathOperator{\Coh}{Coh}

\DeclareMathOperator{\Hom}{Hom}

\DeclareMathOperator{\Ker}{Ker}

\DeclareMathOperator{\Aut}{Aut}

\DeclareMathOperator{\Lie}{Lie}

\DeclareMathOperator{\const}{const}
\DeclareMathOperator{\Ell}{Ell}

\DeclareMathOperator{\chr}{char}
\DeclareMathOperator{\pt}{pt}
\DeclareMathOperator{\cochar}{cochar}

\DeclareMathOperator{\Pic}{Pic}
\DeclareMathOperator{\tr}{tr}

\DeclareMathOperator{\Spec}{Spec}

\DeclareMathOperator{\supp}{supp}
\DeclareMathOperator{\ev}{ev}

\DeclareMathOperator{\QM}{\mathsf{QM}}
\DeclareMathOperator{\Maps}{\mathsf{Maps}}
\DeclareMathOperator{\Jet}{\mathsf{Jet}}
\DeclareMathOperator{\jet}{\mathsf{jet}}
\DeclareMathOperator{\End}{End}
\DeclareMathOperator{\Image}{Image}

\DeclareMathOperator{\diag}{diag}

\DeclareMathOperator{\chern}{ch}

\DeclareMathOperator{\Desc}{Desc}
\DeclareMathOperator{\Symm}{Symm}
\DeclareMathOperator{\Vertex}{\textup{\textsf{Vertex}}}
\DeclareMathOperator{\VwD}{\textup{\textsf{VwD}}}

\newcommand{\xrightarrowdbl}[2][]{%
  \xrightarrow[#1]{#2}\mathrel{\mkern-14mu}\rightarrow
}

\newtheorem{Theorem}{Theorem}

\newtheorem{Lemma}{Lemma}[section]
\newtheorem{Proposition}[Lemma]{Proposition}
\newtheorem{Corollary}[Lemma]{Corollary}

\theoremstyle{definition}

\newcommand{\Mbar}{\overline{M}}

\begin{document}

\title{Nonabelian stable envelopes, vertex functions with descendents, and integral
  solutions of $q$-difference equations} 
\author{Andrei Okounkov} 
\date{}
\maketitle

\abstract{We generalize the construction of elliptic stable envelopes
  \cite{ese, part1} to actions of connected reductive groups and give a direct inductive proof of
  their existence and uniqueness in a rather general situation. We
  show these have powerful enumerative applications, in particular,
  to the computation of vertex functions and their monodromy.} 

\setcounter{tocdepth}{2}
\tableofcontents

\section{Introduction}

\subsection{Interpolation and stable envelopes}

\subsubsection{}

The univariate interpolation formula
$$
f(x) = \sum_{i=1}^n f(a_i) \prod_{j\ne i} \frac{x-a_j}{a_i-a_j} \,, 
$$
which dates back to the late 18 century, gives a canonical splitting of the
surjection 
\begin{equation}
\xymatrix{
\Z[x,a_1,\dots,a_n] \ar[rr] \ar@{=}[d]&& \Z[x,a_1,\dots,a_n] \big/ \left( \prod_i
  (x-a_i) \right) \ar@{=}[d]\\
\Hd_\bGt(\C^n) \ar[rr]^{\textup{restriction\qquad\qquad\qquad}} && \Hd_{\bGt}(\C^{n} \setminus \{0\}) =
\Hd_{\bGa}(\bP^{n-1}) \,. \ar@/_1.64pc/[ll]_{\textup{interpolation}}
} \label{Lagrange}
\end{equation}
The bottom row in \eqref{Lagrange} provides one possible topological
interpretation of this elementary fact.
The groups there are equivariant cohomology groups for
the groups (in this instance, tori) 
\begin{equation}
1 \to \bG \to \bGt \to \bGa \to 1 \,, \label{bGt}
\end{equation}
where $\Lie \bGt$ acts on $\C^n$ by 
$$
(x,a_1,\dots,a_n) \mapsto \diag(a_1 -x, \dots, a_n-x) 
$$
and rank 1 subtorus $\bG \subset \bGt$ corresponds to the variable $x$.

Since $0\in \C^n$ is precisely the $\bG$-unstable locus in $\C^n$ for
a certain linearization of $\cO_{\C^n}$, \eqref{Lagrange} can be
interpreted as an instance of a canonical lift of cohomology classes
from a GIT quotient $\bP^{n-1} = \C^n \rdd \bG$ to the ambient
quotient stack $\left[\C^n / \bG\right]$. \footnote{For the Grassmannian $\Gr(k,n) = \Hom(\C^k,
\C^n) \rdd GL(k)$, the problem  becomes to interpolate a symmetric polynomial 
$f(x_1,\dots,x_k)$ from its values at $\binom{n}{k}$ points where
$x_i=a_{\mu_i}$ with $\mu_1 < \dots < \mu_k$. 
In the simplest example of general theory worked out in
Appendix \ref{s_A}, the reader will recognize in the elliptic stable
envelope \eqref{eq:29} a class that plays the role of the improved
conormal to a
Schubert cell in in $T^*\Gr(k,n)$.}

\subsubsection{}
Parallel formulas exist for interpolation of Laurent polynomials,
namely
\begin{equation}
  \label{eq:9}
  f_L(x) =
    \sum_{i=1}^n f(a_i) \left({x}\big/{a_i}\right)^L \, \prod_{j\ne i}
    \frac{1-x/a_j}{1-a_i/a_j} \,, \quad L \in \Z\,, 
\end{equation}
as well as for sections of line bundles on elliptic curves: 
\begin{equation}
  \label{ellLagr} 
   f_z(x) =
    \sum_{i=1}^n f(a_i)  \frac{\vartheta(z x /a_i)}{\vartheta(z)}
    \prod_{j\ne i} \frac{\vartheta(x/a_j)}{\vartheta(a_i/a_j)}\,,
    \quad z \in E, 
  \end{equation}
  and these have an interpretation in equivariant K-theory and
  equivariant elliptic cohomology, respectively.

  In \eqref{ellLagr},
  we write the group law on an elliptic curve $E$ multiplicatively and
  denote by $\vartheta(x)$ the unique section of the line bundle with
  divisor $1\in E$. Line bundles on elliptic curves have moduli, and
  \eqref{ellLagr} is a section of a line
  bundle $\cS_z$ of degree $n$ corresponding to the point
  $$
  z/\textstyle{\prod} a_i \in E \cong \Pic_n(E) \,.
  $$
The pole at $z=1$ in \eqref{ellLagr} means that interpolation fails
for special \emph{resonant} values of $z$, which concretely means
the section $\prod \vartheta(x/a_i)$ of $\cS_1$ vanishes at all
interpolation points. 

As the elliptic curve generates to a nodal rational curve, the poles
in $z$ lead to a piecewise constant behavior in \eqref{eq:9},
producing Laurent polynomials with
$$
\deg  f_L(x) = \textup{Newton polygon} (f_L) = [L,L+n-1] \,.
$$
This can be phrased as $f_L$ being sections of differently linearized
line bundles on the \emph{compactified} K-theory
$\overline{K}_\bGt(\C^n)$ as in Section 3.3 of \cite{part1}.

\subsubsection{}

This baby example illustrates the general principles that
\begin{enumerate}
\item[(1)] elliptic cohomology
  classes are section of line bundles $\cS$ that have moduli;
  \item[(2)]  the interpolation, which in this paper will mean a lift from a GIT
quotient to the corresponding quotient stack, has a unique solution for a
certain Zariski open\footnote{but not dense, e.g. we need $\deg \cS_z
  = n$ above} set of $\cS$;
\item[(3)] the resonant divisors in the moduli of $\cS$ will produce a
certain characteristic wall-and-chamber pattern upon specialization to
K-theory, compare with \cites{DHLMO, HLS}. 
\end{enumerate}



\subsubsection{}

The general interpolation problem considered in this paper is the
following. We consider an action of a connected reductive 
group $\bG$ on a smooth quasiprojective variety $\bX$. Our 
main assumption about $\bX$ is that the quotient stack $\left[ \bX/\bG
\right]$ has a polarization $T^{1/2}\left[ \bX/\bG
\right]\in K([\bX/\bG])$, see Section \ref{defpol}. 

With some usual assumptions about
total attracting sets, we prove in Theorem \ref{t1} that the interpolation problem has a
unique solutions for an open dense set of line bundles $\cS$
algebraically equivalent to $\Theta(T^{1/2}\left[ \bX/\bG
\right])$. Here and below we follow the notations of the Appendix to \cite{part1}
for various elliptic cohomology constructions. In particular, the
elliptic Euler class of a vector bundle $V$ on $\bX$ is a section of a certain
line bundle on $\Ell(X)$ denoted $\Theta(V)$. 

A more delicate result, required for many application, is the analog
of Theorem \ref{t1} for  $\left[ \bX\rd \bG
\right]= \left[ \mu^{-1}(0)/\bG\right]$ when the action of $\bG$ on
$X$ is symplectic with a moment map $\mu$. See Theorem \ref{t2} in 
Section \ref{s_symplr}, where the interpolation statement is phrased
as a result on $\bX$ for elliptic cohomology classes supported on
$\mu^{-1}(0)$.

Our constructions are always equivariant with respect to an
 ambient group of automorphisms as in \eqref{bGt}.

 \subsubsection{}
Among earlier results of this general flavor, one should mention the
following.

First, when $\bG$ is a abelian, this is essentially the theory of elliptic
stable envelopes, initiated in \cite{ese} and revisited in broader 
generality in \cite{part1}. For this reasons, we call the
interpolation map the \emph{nonabelian} stable envelope.

Second, there is a special case of nonabelian $\bG$ which is already covered
by different existing techniques. Namely, if $\bG = \prod GL(n_i)$ then
the construction of Section 2.1 of \cite{Bethe} extends verbatim to
elliptic cohomology and reduces nonabelian stable envelopes to the abelian
ones, see also Section \ref{s_GL} below.

Our interpolation results may be seen as an elliptic
cohomology parallel to the K-theory and $\Db \Coh$ constructions of
\cites{DHLMO, HLS}, many important ingredients of which originated in the solo
work \cite{DHL} of D.~Halpern-Leistner. The setting of elliptic cohomology offers several
strong technical advantages, as the reader will notice. The
categorification of our elliptic cohomology construction is an area of
active current research, see e.g.\ \cite{AO2}.

\subsection{Enumerative geometry applications}

\subsubsection{}

The second principal result of this paper is an 
application of elliptic stable envelopes to the enumerative theory of
maps $C \to [\bY/\bG]$, where $C$ is curve,
smooth or nodal, and $\bY$ is either smooth or of the form $\bY =
\mu^{-1}(0)\subset \bX$. For concreteness, we will focus on the latter
case.

A map is called \emph{stable} if it takes unstable values only at 
finitely many nonsingular points of $C$.  For historical reasons,
these are called stable quasimaps and form a nice substack denoted by
$$
\QM(C\to \bY\rdd\bG) \subset \Maps(C \to  [\bY/\bG]) 
$$
see \cite{CFKM} and references therein, as well as a brief
recollection in Section \ref{s_quasimaps} below. It has a perfect
obstruction theory that yields a certain canonical K-theory class
$\tO_\vir$, called the symmetrized virtual structure sheaf.

\subsubsection{}

Many problems at the core of modern enumerative
geometry can be phrased as computations with $\tO_\vir$ for
suitable
$C$ and $\bX$. For instance, Hilbert schemes of
points, and more general moduli spaces of sheaves for ADE surfaces are
examples of Nakajima quiver varieties \cite{Nak1}, and so can
be written in the form $[\bX\rd\bG]$, where $\bG=\prod GL(n_i)$ and
$\bX$ is a linear representation of $\bG$. The corresponding quasimap
moduli spaces capture the K-theoretic Donaldson-Thomas counts for 
threefolds fibered in ADE surface, and provide some of the key
computations in the general DT theory of threefolds, see
\cite{Takagi}.

\subsubsection{}
Computations with $\tO_\vir$  really pivot on being able to
do pull-push in the following diagram of evaluation maps 
\begin{equation}
  \label{eq:18_}
  \xymatrix{
    & \ev_\infty^{-1}( \textup{stable}) \ar[ld]_{\ev_0}
    \ar[rd]^{\ev_\infty}  \ar@{^{(}->}[rr]&&\QM(\bP^1
    \xrightarrow{\,\,f\,\,} \bY\rdd \bG)  \\
    [\bX/\bG] && \bX\rd\bG\,, 
   } 
\end{equation}
where $ev_p$ denote the evaluation at $p\in\bP^1$. While the map
$\ev_0$ is not proper in \eqref{eq:18_}, it is proper if one fixes the
degree of $f$ and restricts to maps that are fixed under the action of
$$
\Aut(\bP^1,{0,\infty}) \cong \Ct_q \,. 
$$
Here the subscript means that we denote an element of this group by
$q$. So, the operator
\begin{equation}
  \label{defbJ2}
\VwD: K_{\bGa}(\bX\rd \bG) \xrightarrow{\quad} K_{\bGt\times \Ct_q}(\bX)_{\textup{localized}} [[z]]
\end{equation}
given by 
\begin{equation}
  \label{defbJ}
  \VwD = \ev_{0,*} \left( z^{\deg f} \tO_\vir \otimes
  \ev_\infty^*(\, \cdot \,) \right) \,. 
\end{equation}
is well-defined. The abbreviation $\VwD$ stands for \emph{vertex with descendents}, this is an object
studied at length in \cite{Bethe,pcmi,Smir_desc}. Its limit in
ordinary equivariant cohomology is known under various names,
including the I-function, and is the subject of a really vast
literature, see for instance \cite{Iri} and references therein.

\subsubsection{}

In this paper, we \emph{identify} descendent vertices with the
nonabelian stable envelopes for $E=\Ct/q^\Z$. In Theorem \ref{t3}, we prove, with usual
technical assumptions, the commutativity
of the following diagram
\begin{equation}
  \label{eq:2}
  \begin{tikzcd}
    [%
    ,row sep = 7ex
    ,/tikz/column 1/.append style={anchor=base west}
    ,/tikz/column 6/.append style={anchor=base west}
    ]
    K_{\bGa}(\bX\rd \bG)_{z,\mero} \ar[d,swap,"\bGamma"] 
    \arrow[rrrrr,"\textup{ch(elliptic stable envelope)}"]
    &&&&&
K_{\bGt}(\bX) _{z,\mero}   \ar[d,"\hbar^{\dots} \otimes \bGamma' "]  \\
 K_{\bGa}(\bX\rd \bG)_{z,\mero} \arrow[rrrrr,"\VwD"]
&&&&&  K_{\bGt}(\bX) _{z,\mero}  \,, 
   \end{tikzcd}  
\end{equation}
in which:
\begin{itemize}
\item[$\bullet$] $ K_{\bGt}(\bX)_{z,\mero}$ denotes meromorphic
  functions on
  \begin{equation}
  \Spec K_{\bGt}(\bX) \times \{ |q| < 1\} \times \{ 0 < \textup{distance}(z,0_{\Lamp})<
  \varepsilon\}
  \label{spec_stack}
\end{equation}
  with poles of the form discussed in Section \ref{s_mero} union
  the resonant locus for the K\"ahler variables $z$ as in Section 2.3
  on \cite{part1}. The ample line bundle $\Lamp\in \Pic(X)$ is used
  here to define stability and the point $z=0_{\Lamp}$ is the origin
  in the expansion \eqref{defbJ}. 
 \item[$\bullet$]  The Chern-character-type map \eqref{chern} applied to
   the elliptic stable envelopes yields the top arrow in \eqref{eq:2}.
\item[$\bullet$]  The vertical arrows in \eqref{eq:2} are multiplications by certain
  characteristic classes of $\bX$ built from $q$-Gamma functions in a
  way that parallels the work of H.~Iritani and others in cohomology,
  see \cite{Iri}. As explained in Section \ref{s_Dd}, it is much more
  natural to consider these factors in the setting of equivariant
  K-theory and $q$-difference special functions. 
\end{itemize}

\subsubsection{}

The following philosophical conclusion can be made from \eqref{eq:2}.
To extend an elliptic cohomology class from the stable locus
to all of $[\bX/\bG]$, one can do K-theoretic integration over all
maps $f:[\C/\Ct] \to [\bX/\bG]$ such that $f$ takes a stable value at
the generic point.

To be sure, since such interpolation is unique, the primary flow of
information is in the opposite direction, that is, from elliptic
stable envelopes to enumerative geometry. 


\subsubsection{}

In order to have a commutative diagram of the form \eqref{eq:2},  it is
essential to work in equivariant elliptic cohomology. Indeed, the
bottom arrow in \eqref{eq:2} is an analytic function of the K\"ahler
variables $z$, and therefore the top arrow must depend on the same
variables $z$ also analytically.

We recall that stable envelopes
in equivariant K-theory depend on the corresponding parameter in a
piecewise constant way, while stable envelopes in equivariant
cohomology do not depend on it at all.

\subsection{Integral solutions of $q$-difference equations} 

\subsubsection{}
Theorem \ref{t3} implies a certain integral representation for the
vertex function with descendents. Using the equivalence between
relative and descendent insertions established in \cite{Bethe}, one
obtains from it an integral representation for vertices with relative
insertions, see \cite{pcmi, Takagi} for a detailed introduction.

It is known, see \cite{pcmi,OS} that these relative counts are
fundamental solutions of certain $q$-difference equations of
general importance in mathematical physics, including the quantum
Knizhnik-Zamolodchikov equations, and their relatives. Thus we obtain
an integral solutions of these equations.

\subsubsection{}
We recall that integral representations of solutions
of $q$-difference equations generalize the eigenvalue
problem for their $q\to 1$ limit, and thus the subject generalizes
the search for eigenvalues and eigenvectors in the vast subject broadly
known as Bethe Ansatz, see the discussion in \cite{Bethe} and
references therein.

We admit it is a hopeless task within the scope of this paper to supply an overview or a
representative list of references on the subject Bethe Ansatz. The
subject is revisited by every generation of mathematical
physicists, the present generation taking the inspiration from the 
insights of Nekrasov and Shatashvili \cite{NS}.  See, however,
\cites{Matsuo, Reshet_int,Varchenko1} for several early references dealing with integral solutions
of the qKZ equations.

\subsubsection{}

The integral we get are of the so-called Mellin-Barnes type.  Their
schematic form
\begin{equation}
  \label{eq:5}
  \left( \alpha, {\textstyle{\begin{bmatrix}
    \textup{\scriptsize{fundamental}} \vspace{-2mm} \\ 
    \textup{\scriptsize{solution}} 
  \end{bmatrix}}}\beta\right)= \frac{1}{|W|} \int_{|x_i|=1}
  \mathbf{f}_\alpha(x,\dots) \,  \mathbf{g}_\beta(x,\dots)\,
  \bPhi(x,\dots)
  \prod \frac{dx_i}{2\pi i x_i}
\end{equation}
is exactly the same as discussed in Section 1.1.6 of
\cite{Bethe}. Here:
\begin{itemize}
\item[$\bullet$] $\alpha,\beta$ are vectors in linear space on which the
  difference operators act, or more precisely vectors in the fiber of
  the $q$-difference connection over a point fixed by
  $q$-shifts. Here, this vector space is identified geometrically with
  $K_{\bGa}(\bX\rd \bG)$. 
\item[$\bullet$] $\alpha \mapsto \mathbf{f}_\alpha(x,\dots)$ and
  $\beta \mapsto \mathbf{g}_\beta(x,\dots)$ are linear maps to
  functions of $x$ and other variables. Geometrically, these 
  are functions on \eqref{spec_stack}. 
 \item[$\bullet$] The variables $x_i$ parametrize a maximal compact
   torus in $\bG$ and $W$ is the corresponding Weyl group. 
 \item[$\bullet$] $\bPhi(x,\dots)$ is a product of $q$-Gamma functions
   in some monomials, geometrically identified as certain
   weights of $\bGt$. 
 \end{itemize}
 It was shown in \cite{Bethe} that the correct choice for
 $\mathbf{f}_\alpha$ is the stable envelope in equivariant K-theory
 and it was also noted that the stable envelopes in equivariant elliptic
 cohomology should be the natural choice for $\mathbf{g}_\beta$.
 As a corollary of our Theorem \ref{t3}, this expectation is confirmed
 in full generality considered here.

 As a remark, formula 4 in \cite{Bethe} contains an elementary
 factor denoted by $\mathbf{e}(x,z)$. Here the elliptic automorphy induced by
 this factor is part of the
 definition of $\mathbf{g}_\beta$. 

 \subsubsection{}
 Our proof of Theorem \ref{t3} and, hence of the integral
 representation of the vertex functions, uses purely geometric
 arguments and \emph{does not} rely on contour deformations,
 residue computations, and other traditional techniques in the
 subject.

 However, to help the reader parse the proof of Theorem \ref{t3}, we have provided
 an almost word-by-word translation into such language in the simplest
 example of vertex functions for $T^*\Gr(k,n)$. This can be found
 in Appendix \ref{s_A}. This appendix is
 recycled from notes, dating back to the Summer of 2014, that did not make it to final version of
 \cite{ese}.

 \subsection{Monodromy} 

 \subsubsection{}
Monodromy of the quantum differential and difference equations is
a subject on the crossroads of many fields, see e.g.\ \cites{BezOk,
  slc, icm}. One of the principal reasons for the introduction of elliptic stable
envelopes in \cite{ese} were precisely their applications to monodromy
in 
enumerative problems.

We recall
that in \cite{ese} the monodromy of the quantum difference equations
in \emph{equivariant variables} was computed in terms of elliptic R-matrices
introduced in \cite{ese}. By construction, these R-matrices are ratios
of two abelian stable envelopes \cite{ese, part1}.

Since the $q$-difference equations in question generalize,
in particular, the quantum Knizhnik-Zamolodchikov equations of
I.~Frenkel and N.~Reshetikhin \cite{FrenResh}, this result encompasses a lot of prior
research, the introduction to which may be found in \cite{EFK}.

\subsubsection{} 

In this paper, as a corollary of Theorem \ref{t3}  we compute the
monodromy in the \emph{K\"ahler variables} as the ratio of two
nonabelian stable envelopes, see Corollary \ref{c_Monodr} in
Section \ref{s_c_Monodr}. 

Recall that the monodromy of a $q$-difference equation is defined as the
ratio of the fundamental solutions at two different points fixed by
$q$-shifts. In the present setting, there are points $z=0_{\cL_{\pm}}$
at infinity of the torus $\bZ$ of K\"ahler variables 
labelled by two different ample
bundles $\cL_{\pm} \in \Pic_\bG(\bX)$. As usual, the R-matrix is
defined as the ratio of the corresponding stable envelopes, and we
prove that it equals the monodromy.

 \subsubsection{}
 In the $q\to 1$ limit, the quantum difference equations become
 the quantum differential equations, and the functions
 $\mathbf{f}_\alpha$ and $\mathbf{g}_\beta$ turn into their
 cohomological and K-theoretic analogs, respectively.

 Integral solutions and monodromy of quantum differential
 equations is a subject that is very closely linked, in particular, to the work of
 Kentaro Hori on \emph{grade selection rules}, see e.g.\ 
 \cites{Hori_Beij, Hori_Tong}. These grade selection rules are a form
 of stable envelopes in derived categories of coherent sheaves and
 equivariant K-theory, and they have influenced, in particular, the work
 of Danial Halpern-Leistner \cite{DHL}, as well as many other advances
 and computations.

 Our proof of Theorem \ref{t3} gives a uniform general treatment of
 these and many other examples of integral solutions of quantum
 differential and difference equations found earlier in the physics
 literature. 
 
 \subsection{Acknowledgements}

 I'd like to reiterate the words of gratitude to many people from \cite{part1}, and in particular, I'd like to thank Mina
Aganagic, Davesh Maulik, and Daniel Halpern-Leistner, for the
inspiration that their work (including joint work) provided for
the present project.

I am grateful to the Simons Foundation for being supported as a Simons
Investigator. I thank the
Russian Science Foundation for the support by the grant  19-11-00275.





\section{Nonabelian stable envelopes}\label{s_GIT}

\subsection{Restriction to the semistable locus}

\subsubsection{}

The constructions of \cite{part1} were for a torus $\bA$ acting on a smooth
quasiprojective variety $\bX$ over $\C$. They admit the following
generalization for actions of connected reductive groups $\bG$. The
setup follows the definition of the categorical and K-theoretical
stable envelopes for nonabelian actions, see \cite{DHLMO}. We denote by
$\bA$ a maximal torus of $\bG$. 

\subsubsection{}

As before, we work equivariantly with respect to an ambient group of
automorphisms. In the abelian case, this meant that we considered a larger torus $\bT
\supset \bA$ acting on $\bX$. In the nonabelian case, we consider a
larger group \eqref{bGt}, where $\bGa$ may be an arbitrary reductive group. 

\subsubsection{}

Let $\Lamp$ be an ample $\bGt$-equivariant line bundle on $\bX$. We may
assume that $\bX$ is $\bGt$-equivariantly embedded in $\bP(V)$, where
$V$ is a $\bGt$-module and $\Lamp=\cO_{\bP(V)}(1)$.

\subsubsection{}
In \cite{part1}, it was assumed that the total attracting set
\begin{equation}
\Attr_\sigma= \{ x \in \bX, \,\, 
\textup{$\lim_{t\to 0} \sigma(t) \cdot x$
  exists}\}\label{Attrs}
\end{equation}
is closed 
for all $1$-parameter subgroups
$$
\sigma: \Ct \to \bG
$$ 
in a certain cone in $\cochar(\bA)$.

In the nonabelian case, one may impose this condition for
$1$-parameter subgroups in a suitable $W$-invariant cone in
$\cochar(\bA)$. For simplicity, and with
concrete applications in mind, we \emph{assume} that the total attracting
set \eqref{Attrs} is closed for \emph{all} $\sigma$. 

\subsubsection{}\label{defpol} 
Given a K-theory class $\cV\in K_\bG(\bX)$, we will call its
polarization $\cV^{1/2}$ any solution of the equation 
$$
\cV^{1/2} + \left( \cV^{1/2} \right)^\vee = \cV\,. 
$$
In particular, we take as $\cV$ the tangent bundle to the quotient stack 
$$
 T \left[\bX/\bG\right] = T \bX - \fg_\C \,, 
$$
with the adjoint representation of $\bG$ on the second summand. In
what follow, we abbreviate $\fg_\C$ to $\fg$, as all vector bundles in
our context are complex.

We 
assume  $T \left[\bX/\bG\right]$
has a polarization which we denote by $\Ths$, or $T^{1/2}$
for brevity. One example of such situation is when $\bG$ acts on a
smooth quasiprojective $\bX'$ which has a polarization and $\bX = \bX'
\times \fg$. 

We assume that $\Ths$ has a lift to an element of
$K_{\bGt}(\bX)$, which is automatic if 
$\bGa$ is a connected factorial group, see \cite{Merk}. 

\subsubsection{}
Recall that a nonzero vector $v \in V\setminus 0 $ is called unstable
under the action of $\bG$ if $0 \in
\overline{\bG v}$. Clearly, this is invariant under dilation and hence 
defines the unstable locus $\bP(V)_\ust \subset \bP(V)$. One sets 
$$
\bX_\ust = \bX \cap \bP(V)_\ust
$$
and defines the semistable locus as the complement
$$
\bX_\sst = \bX \setminus \bX_\ust \,. 
$$
All these loci are $\bGt$-invariant. 

\subsubsection{}\label{s_res1} 
Consider  the functorial maps
\begin{equation}
  \label{st2X}
  \xymatrix{\Ell_\bGt(\bX_\sst) \ar[r]^{\iota_\sst} \ar[dr]_{p_\sst} & \Ell_\bGt(\bX)\ar[d]^{p} \\
    & \Ell_{\bGa}(\pt)}
\end{equation}
and the line bundle 
\begin{equation}
  \label{cSG}
  \cS = \Theta(\Ths) \otimes \bigotimes \cU(\cL_i,z_i)
\end{equation}
defined in both the source and the target of $\iota_\sst$. Here
$\cL_i$ are equivariant line bundles on $\bX$ with $\cL_1 = \Lamp$ and
the coordinates
$$
z_i \in E=\Ell_{U(1)}(\pt)
$$
are added to the base $\bB$ of the
elliptic cohomology as in Section 2.3.8 of \cite{part1}.

We will be
concerned with the restriction of section of $\cS$ from $\bX$ to the
stable locus $X_\sst$. This map is may be viewed as the inverse of the
stable envelope, and so we define
\begin{equation}
  \Stab^{-1} = \iota_\sst^*: \quad 
p_* \cS \to p_{\sst,*} \cS \,. 
  \label{Stabm1}
\end{equation}
Recall that stable envelopes are unique, but may have
poles. Reflecting this, the main
result of this Section is the following: 

\begin{Theorem}\label{t1}
The map $\Stab^{-1}$ is injective, with a torsion cokernel. 
\end{Theorem}

\subsubsection{}

It is interesting to compare the ranks of $p_{\sst,*}
\cS$ and $p_* \cS$. If $\bG$ acts on $\bX_\sst$ with finite 
stabilizers, then the fibers of $p_\sst$ in \eqref{st2X} are
$0$-dimensional.
For instance, if $\bG$-action on $\bX_\sst$ is free, then
$p_{\sst,*} \cS$ has the same $(\bZ/2)$-graded rank as $\Hd(X\rdd
\bG)$.

On the other hand, typically, the fibers of $p_*$ are quotients of
abelian varieties by a finite group and the bundle $\cS$ is relatively
ample.  Thus the rank $p_* \cS$ may be computed from, essentially, the
top self-intersection of the divisor $\Theta(\Ths)$ on
$\Ell_\bG(\pt)$. 

In this connection, we note the following. First, while in this paper we
are emphasizing a geometric language and geometric
techniques, the phenomena we discussed are related to the fact that 
cohomological computations on $\bX\rdd \bG$ may be matched to residue computations in
certain integrals, see the discussion in Appendix \ref{s_A}. Second, in the bulk of
the paper, we will be concerned with sections of $\cS$ that have 
restricted support, as in \eqref{cSG3}. The sheaf in $\cS_\symr$ in
\eqref{cSG3} is not a line bundle and it is not so straightforward to
count its sections. Translated into the language of Mellin-Barnes
integrals, this means that certain poles don't correspond to
cohomology classes of $\bX\rd\bG$ because we force the integrand to annihilate the corresponding residues.




\subsubsection{}
A special case of Theorem \ref{t1} is when the stable locus is empty,
which thus implies $p_* \cS = 0$. The stable locus may be empty for
a trivial reason, namely if a connected subgroup of $\bG$ acts
trivially on $\bX$, but nontrivially on $\Lamp$. In particular, this
happens with the hypothesis of the following 

\begin{Lemma}\label{l1}
Suppose there exists
$$
\sigma: \Ct \to \textup{center}(\bG)
$$
which acts trivially on $\bX$. Consider
$$
\cS = \Theta(\cV) \otimes \bigotimes \cU(\cL_i,z_i) \,, 
$$
where $\sigma$ acts trivially on a 
vector bundle $\cV$ over $\bX$, but with certain weights
$n_i$ on line bundles $\cL_i$, not all of which are zero.
Then  $p_*
\cS = 0$ and $\bR^i p \cS$ is supported on the locus 
\begin{equation}
\sum n_i z_i = 0 \in E \,.\label{sumz}
\end{equation}
\end{Lemma}

\subsection{Proof of Theorem \ref{t1}}

\subsubsection{Proof of Lemma \ref{l1}}

The consider the pushforward of $\cS$ under the map
$$
p_\sigma: \Ell_\bGt(\bX) \to \Ell_{\bGt/\sigma}(\bX)\,, 
$$
which is a fibration with fiber $E$. The bundle $\Theta(\cV)$ is 
trivial on the fibers of $p_\sigma$, while the bundle $\bigotimes\cU(\cL_i,z_i)$ has
degree $0$ and nontrivial away from \eqref{sumz}. Thus already $\Rd
p_\sigma \cS$ has the required support and the lemma follows. 



\subsubsection{} \label{s_Xi_induct}

Consider the stratification of the unstable locus, as constructed in
\cites{Bogo,Hess,Kempf,Ness,Rous}, see e.g.\ Chapter 5 in \cite{VinPop} for an exposition. It takes
the form
\begin{equation}
\bX_\ust=\bX_1 \supset \bX_2 \supset \bX_3 \dots\label{XiX}
\end{equation}
On each step of \eqref{XiX}, we have
\begin{equation}
  \bX_i \setminus \bX_{i+1} \cong \bG \times_{\sP_i}
  \Attr_{\sigma_i}(F_i)\label{XXG}\,, 
\end{equation}
where
\begin{enumerate}
\item[(1)] $\sigma_i: \Ct \to \bG$ is a one-parameter subgroup,
\item[(2)] $\sP_i \subset \bG$ is a parabolic subgroup such that its Lie
  algebra 
  $$
  \fp_i = \Attr_{\textup{Ad}(\sigma_i)} \subset \fg 
  $$
  is the subspace of nonnegative weights with respect to the
  adjoint action of $\sigma_i$, 
\item[(3)] $F_i$ is an open subset of $\bX^{\sigma_i}$ and $\sigma$
  acts with a \emph{negative} weight on $\Lamp\big|_{F_i}$. 
\end{enumerate}

\subsubsection{}

We can factor the restriction from $\bX$ to $\bX_\sst = \bX \setminus
\bX_1$ into a sequence of restrictions from $\bX\setminus \bX_{n+1}$
to $\bX\setminus \bX_{n}$. Thus, by induction, we may assume that
$\bX_2 = \varnothing$.

Consider the diagram 
\begin{equation}
\xymatrix{\bX_1 = \bG \times_{\sP_1}
  \Attr_{\sigma_1}(F_1) \ar[rr]^{\qquad \qquad j} \ar[d]_\pi  &&\bX \\
 \bG \times_{\sP_1}
  F_1 } \,,\label{XpiY}
\end{equation}
in which $\pi$ is a fibration in affine spaces and, in particular, an
isomorphism in elliptic cohomology. In \eqref{XpiY}, $\sP_1$ acts on $F_1$ via the
projection to its Levi factor $\bL_1$.

\subsubsection{}

Observe that 
$$
N_{\bX/\bX_1} = N_{\bX/F_1,<0} - \fg_{<0}\,,
$$
where subscripts indicate repelling weights for $\sigma_1$.

Consider the decomposition
$$
T^{1/2}\big|_{F_1} = T^{1/2}_{F_1} + T^{1/2}_\textup{moving}
$$
into fixed and moving parts for the action of 
$\sigma_1$. Up to
duals and weights of $\bGa$, the moving part $
T^{1/2}_\textup{moving}$ coincides with
$N_{\bX/\bX_1}$. Therefore, see the discussion in Section 2.2 of
\cite{part1}, we have 
$$
\Theta(-N_{\bX/\bX_1}) \otimes j^*\cS = \pi^* 
\Theta(T^{1/2}_{F_i}) \otimes \bigotimes \cU(\cL_i, z_i') \otimes
\dots \,, 
$$
where dots denote a line bundle pulled back from
$\Ell_\bGa(\pt)$, the set $\{\cL_i\}$ may have become larger, and 
the K\"ahler variables include shifts by weights of $\bGa$.

\subsubsection{}

As in Section 2.5.7 of \cite{part1}, we have an exact sequence 
$$
0 \to \Theta(-N_{\bX/\bX_1}) \to \cO_{\Ell_\bG(\bX)} \to
\cO_{\Ell_\bG(\bX\setminus \bX_1)} \to 0 \,. 
$$
It follows from Lemma \ref{l1} that, after we tensor this exact
sequence with $\cS$,  we get
$$
0 \to  \Theta(-N_{\bX/\bX_1}) \otimes \cS \to \cS_{\Ell_\bG(\bX)} \to
\cS_{\Ell_\bG(\bX\setminus \bX_1)} \to 0 \,,
$$
in which the kernel has no $\bR^0p$ and torsion $\bR^1p$. This
completes the induction step and the proof of the theorem.

\subsection{Symplectic reductions}\label{s_symplr} 

\subsubsection{}

We now suppose that $\bX$ is a holomorphic symplectic variety and
$\bG$ acts on it
with a holomorphic moment map
\begin{equation}
  \label{mu}
  \mu: \bX \to \fgb^\vee \, .
\end{equation}
We assume $\mu$ is equivariant for a certain linear action of $\bGt$ on
its target and we use boldface symbol $\fgb$ to indicate the required
$\bGt$-module structure in the target. For instance, if $\bGt$ scales
the symplectic form with weight $\hbar$ then
\begin{equation}
  \label{fgbhbar}
  \fgb = \hbar \otimes \fg \,, 
\end{equation}
where $\bGt$ acts on
$
\fg \triangleleft \Lie(\bGt)
$
by the adjoint representation. In practice, this is the important
case.

\subsubsection{}

{}From definitions, 
\begin{equation}
  \label{immu}
  T_x \bX \xrightarrowdbl{\quad d\mu \quad} \fg_x^\perp \subset \fg^*\,, 
\end{equation}
where $\fg_x\subset \fg$ denotes the stabilizer of $x$, $\fg_x^\perp$
its annihilator in $\fg^*$, and the map in \eqref{immu} is surjective.

We assume that there are no strictly semistable points in
$\mu^{-1}(0)$, that is, we \emph{assume} that $\bG$ acts on
$\mu^{-1}(0)_\sst$ with finite stabilizers. Recall that the quotient
$$
\bX\rd \bG = \mu^{-1}(0)_\sst / \bG
$$
is called the algebraic symplectic reduction of $\bX$. With our
assumption, this is a smooth orbifold.
Note that $0\in \fg^*$ is fixed by $\bGa$ and hence this group acts on
$\bX \rd \bG$.

\subsubsection{}

We would like to apply Theorem \ref{t1} to the action of $\bG$ on 
\begin{align}
  \label{eq:1}
  \bY &= \mu^{-1}(0)\,, \\
  T^{1/2} \left[\bY/\bG\right] &= T^{1/2} \bX - \fg_\C \,, 
\end{align}
except $\bY$ is not smooth and so Theorem \ref{t1} does not apply as
stated. Instead, we will push everything forward to $\bX$ and rephrase
Theorem \ref{t1}  as a statement there, with supports on $\bY$. 

\subsubsection{}
As a replacement of the line bundle $\cS$ in \eqref{cSG}, we take the sheaf 
\begin{equation} 
  \label{cSG3}
  \cS_\symr = \Ker \left( \cS^\circ \to \cS^\circ|_{\bX\setminus
    \bY} \right)\,,  
\end{equation}
where
\begin{equation} 
  \label{cSG2}
  \cS^\circ = \Theta(T^{1/2} \bX) \otimes \bigotimes \cU(\cL_i,z_i) 
\end{equation}
is the bundle \eqref{cSG} corrected for the normal bundle
$N_{\bX/\bY}$.

\subsubsection{}

As in Section \ref{s_res1}, we consider the map: 
\begin{equation}
  \Stab^{-1}_\symr  = \iota_\sst^*: \quad 
p_* \cS_\symr  \to p_{\sst,*} \cS_\symr  \,. 
  \label{Stabmmu}
\end{equation}
The sympletic reduction version of Theorem \ref{t1} is the following 

\begin{Theorem}\label{t2}
The map $\Stab^{-1}_\symr$ is injective, with a torsion cokernel. 
\end{Theorem}

\noindent 
We will give separate arguments for injectivity and surjectivity in
Theorem \ref{t2} below. 

\subsubsection{}\label{s_GL}
In the special case
$$
\bG = \prod GL(n_i)
$$
there is a different, and more direct proof of Theorem
\ref{t2}. Namely, the construction given in Section 2.1 of \cite{Bethe}
extends verbatim to elliptic cohomology and, moreover, describes
the map $\Stab$ in Theorem \ref{t2} in terms of stable envelopes for
an auxiliary action of $\Ct$. See also Section 1.3.3 in \cite{Bethe} for an
$R$-matrix formula for the resulting map. 

\subsubsection{Proof of injectivity}

We induct as in Section \ref{s_Xi_induct}. For the induction step, we may assume that $\bX_2 =
\varnothing$. Let $s$ be in the kernel of the restriction to $\bX
\setminus \bX_1$. Then the support of $s$ is contained in
$$
\bX'_1 = \bG \times_{\sP_1}
 \left( \Attr_{\sigma_1}(F_1) \cap \mu^{-1}(\fp_1^\perp)\right)\,. 
 $$
 Since the infinitesimal stabilizer of every point $x\in F_1$ is
 contained in $\fp_1$, we see from \eqref{immu} that $\bX'_1$ is a 
smooth manifold with normal bundle
$$
N_{\bX/\bX'_1} = N_{X/F_1,<0} - \fg_{<0} +\fg^*_{>0} \,. 
$$
Note that the two last terms are dual up to the action of $\bGa$. As
the twists by $\bGa$ do not affect the degree in the variables in
$\bG$, we see that
$$
\deg_{\sigma_1} \Theta(-N_{\bX/\bX'_1}) \otimes \cS_\symr  = 0\,, 
$$
and thus Lemma \ref{l1} applies showing $s=0$.

\subsubsection{Proof of surjectivity}

This will take several steps. In the course of the proof, we will
consider the manifold 
\begin{equation}
\tX = \bX \times \fg \label{def_tX}
\end{equation}
with the adjoint action of $\bG$ on the second factor and 
with a $\bG$-invariant function 
$$
W(x,\xi) = \langle \mu(x), \xi \rangle \,, \quad \xi \in \fg \,. 
$$
These objects are standard in supersymmetric gauge theory literature,
where $\bG$ is a complexification of the gauge group and $\xi$ are
superpartners of the gauge fields.

In gauge theory context, the function $W$ is complemented
by the analogous function $W_\R$ for the real moment map
$$
\mu_\R: \bX \to \fg^*_\comp
$$
and $X\rd \bG$
appears as the critical locus of the combined function
$W+W_\R$ modulo the group of constant gauge transformations, which is
the 
compact form $\bG_\comp\subset \bG$. Compare with Lemma \ref{l_crit} below.

 \subsubsection{}

Consider the semistable locus $\tX_\sst$ and the restriction of the
function $W$ to it. There is the following basic lemma, see
e.g.\ the proof of Lemma 5.4 in \cite{HLS} for the
proof, 

\begin{Lemma}\label{l_crit}
The critical locus of $W$ in the semistable locus $\tX_\sst$ is
$\bY_\sst \times \{0\}$. 
\end{Lemma}

\noindent
We abbreviate $\bY_\sst \times \{0\}$ to $\bY_\sst$. 
By our assumption, the action of $\fg$ is free and the map $d\mu$ is a
submersion in a
neighborhood
of $\bY_\sst$. Therefore, 
$$
N_{\tX/\bY_\sst} \cong \bY_\sst \times \fg \times \fg^* \,, 
$$
and the Hessian of $W$ is the canonical pairing of $\fg$ and $\fg^*$,
in particular, nondegenerate in the normal directions to $\bY_\sst$. Further, the ascending and descending
manifold of
$\Re W$ locally have the form
$$
\{(\xi,\pm \xi^*)\} \subset \fg \times \fg^* \,, 
$$
where the antilinear adjoint
$$
\fg \owns \xi \mapsto \xi^* \in \fg^* 
$$
is with respect ot a $\bG_\comp$-invariant
Hermitian form on $\fg$.

\subsubsection{}\label{s_spp}

Consider the diagram
\begin{equation}
  \label{Desc}
  \xymatrix{\Desc \ar[r]^\iota \ar[d]_\pi & \tX_\sst \\
    \bY_\sst\,,}
\end{equation}
where $\Desc$ is the descending manifold of $\Re W$ inside $\tX_\sst$,
$\iota$ is the inclusion, and $\pi$ is the 
natural projection to the critical locus.

Let $s$ be a section of $\cS_\symr$ on $\bX_\sst$. By assumption it is
a push-forward of a section $s'$ on $\bY_\sst$.  We have
$$
N_{\bX/\bY_\sst} \cong \bY_\sst \times \fg  \cong N_{\tX/\Desc}  \,. 
$$
Therefore $s'' = \iota_\circledast\,  \pi^* s'$ is a section of
$\cS^\circ$, supported on the descending manifold inside the semistable
locus of $\tX$.  We have
$$
T^{1/2} \left[ \tX / \bG \right] = T^{1/2} X \,. 
$$
Therefore, by Theorem \ref{t1}, $s''$ may be extended to a section of
$\cS^\circ$ on the whole of $\tX$, probably with some poles in the
variables $z_i$ and $\bT/\bA$.

Since the function $W$ is $\bG$-invariant, the section $s''$ is
supported on the subset $\Re W \le 0$. 

\subsubsection{}
Consider the map
$$
\bX \owns x \xrightarrow{\quad f \quad} (x,\mu(x)^*) \in \tX\,.
$$
Clearly
$$
W(f(x)) = \| \mu(x) \| \ge 0
$$
and thus $f^* s''$ is supported on $\bY$.

Further, on the semistable
locus, the map $f$ realizes the neighborhood of $\bY_\sst \subset
\bX_\sst$ as the ascending manifold of $\Re W$. Thus 
$$
f^* s'' \big|_{\bX_\sst} = s \,.
$$
and the proof of the surjectivity is complete.

\subsubsection{Remark}\label{s_rmks} 

We note that $\tX$ is an equivariant vector bundle over $\bX$ and thus
any elliptic cohomology class on $\tX$ is a pullback of an elliptic
cohomology class from the base $\bX$. In particular, $s''$ is the
pullback of $f^* s'' = \Stab_\symr (s)$. 

\section{Vertex functions and Monodromy}

\subsection{Quasimaps and vertex functions}\label{s_quasimaps} 

\subsubsection{}
We briefly recall the setup. The reader will find a detailed, or 
a short, 
introduction in \cite{pcmi} and \cite{ese,Bethe}, respectively.

In this section, we take $E=\Ct /q^\Z$ as the elliptic curve of the
elliptic cohomology.

\subsubsection{}

Let $\bG$ be a linear algebraic group acting on algebraic variety 
$\bY$. By definition, a map $f: C \to [\bY/\bG]$ to the corresponding
quotient stack is a principal
$\bG$-bundle $\cP$ over the domain $C$ together with a section
$$
f: C \to \cP \times_\bG \bY \,, 
$$
of the associated $\bY$-bundle over $C$.

Good moduli spaces of \emph{stable} maps $f$ as above may be
constructed if $\bY$ is affine, $\bG$ is reductive, and $C$ is a curve
with at worst nodal singularities, see \cite{CFKM} and references
therein. By definition, $f$ is stable if it evaluates to an orbit in
$\bY_\sst$ at all but finitely many nonsingular points of $\{b_i\}
\subset C$. There points are called the \emph{base points}, or the
singularities, of $f$. Here $\bY_\sst \subset \bY$
is defined using GIT stability with respect to
an ample $\bG$-equivariant line bundle $\Lamp$.

\subsubsection{}

For historical reasons, and to distinguish these object to from
ordinary
maps to the GIT quotient, one calls them \emph{quasimaps}.
We denote by
$$
\QM(C\to \bY\rdd\bG) \subset \Maps(C \to  [\bY/\bG]) 
$$
the corresponding moduli space. We
may 
drop the source $C$, the target $\bY\rdd\bG$, or both from this
notation when they are understood.

\subsubsection{}\label{s_0L} 

The degree of a quasimap may  be measured by pairing it with line
bundles, namely 
\begin{equation}
\langle\deg f, \cL' \rangle  = \deg f^* \cL' \,, \quad \cL' \in
\Pic_\bG(\bY)_\tp \,, \label{Lnum} 
\end{equation}
where the subscript means that we take the quotient modulo the
topological equivalence. Thus the symbol $z^{\deg f}$ is a character of the torus 
$$
\bZ = \Pic_\bG(\bY)_\tp \otimes \Ct  \owns z \,. 
$$
This torus is usually called the \emph{K\"ahler} torus. Generating
functions counting maps of all possible degree are functions (or
formal functions) on $\bZ$. 

The Lie
algebra
$$
\Lie_\R \bZ = \Pic_\bG(\bY)_\tp\otimes\R
$$
contains a fan of cones formed by the images of the ample cones of 
$\bY\rdd\bG$ taken with all possible stability conditions $\cL$. 
The toric variety $\bZb \supset \bZ$ corresponding to this fan 
will be important
below. Every cone, that is, every stability parameter $\Lamp$ defines a point
\begin{equation}
0_\Lamp\in \bZb \,, \label{0L}
\end{equation}
which is fixed by the action of $\bZ$. 

\subsubsection{}\label{sQMass}

Our main object of interest is the case when
\begin{equation}
\bY = \mu^{-1}(0)\label{bY}
\end{equation}
for a Hamiltonian action of a connected reductive group $\bG$ on a
smooth affine algebraic symplectic variety $\bX$. From now on, we
assume that we are in this setting. Additionally, we make the
following simplifying assumptions:
\begin{enumerate}
\item[(1)] We assume that the action of $\bG$ on $\bY_\sst$ is
  free. Probably, actions with finite stabilizers in the semistable
  locus are not much more difficult, but this adds an extra layer of
  complexity to what is already a rather complicated argument.
\item[(2)] We assume that $K_\bGt(\bX)$ generates $K_\bGt(\bY_\sst)$
  by restriction. Such statements, often called 
  hyperk\"ahler Kirwan surjectivity theorems, are known for
  quiver varieties \cite{McGN}. In any event, the image of
  $K_\bGt(\bX)$ is a canonical subspace of $K_\bGt(\bY_\sst)$ and we
  compute the monodromy of the quantum difference equation in that
  subspace.
\item[(3)]
  We assume the polarization $T^{1/2}$ of $\bX$ lifts to a $\bGt$-equivariant
  K-theory class such that 
  $$
  TX = T^{1/2} + \hbar^{-1} \otimes \left( T^{1/2}\right)^\vee \,, 
  $$
  where $\hbar$ is a character of $\bGt$. This means that $\bGt$
  scales the symplectic form with the character $\hbar$ and that 
  $$
  \fgb = \hbar \otimes \fg\,,
  $$
  as $\bGt$-module. This is what happens in the situation of maximal
  interest to us and allows complete matching with the setup of
  \cite{pcmi}
  in what concerns the symmetrized virtual structure sheaves
  $\tO_\vir$ etc. 
\end{enumerate}
We note that the really interesting case is when the character $\hbar$
is nontrivial, because otherwise the quantum difference equation reduces to
its classical part and becomes an equation with constant
coefficients. Constant coefficient equations have trivial monodromy,
up to normalizations.

\subsubsection{} 
An obstruction theory for $\QM$ may be constructed from an obstruction
theory on $\bY$ and the obstruction theory of maps from curves, see
\cite{CFKM}. With the hypotheses of Section \ref{sQMass}, the obstruction
theory is perfect and produces the canonical virtual structure sheaf
$$
\cO_\vir \in K_{\bGa \times \Aut(C)} (\QM) \,. 
$$
Applications in both theoretical physics and enumerative geometry 
dictate the need to work with a \emph{symmetrized} virtual structure
sheaf $\tO_\vir$, which is $\cO_\vir$ twisted by line bundle closely
related to a square root of the virtual canonical line bundle for
$\QM$. See \cite{pcmi} for a detailed discussion. A formula for $\tO_\vir$
appears in \eqref{tOvir} below. 

While the existence of a required square root may be shown on abstract
grounds as in \cite{NO}, it is convenient to pick an explicit sheaf
$\tO_\vir$ using a polarization of $\bX$, see Section 6.1.8  in
\cite{pcmi}.  Since a polarization of $T^{1/2} \bX$ is available in
cases of maximal interest, we will assume that one is fixed.

\subsubsection{}
The moduli space $\QM(C)$ is constructed over the moduli stack of
nodal curves $C$ of a given genus $g$. The curve $C$ may 
may be additionally equipped with nonsingular
distinct marked points $\{p_1,\dots,p_n\}$ none of which is base point
of the quasimap. This gives a map
$$
\pi: \QM_{g,n} \to \Mbar_{g,n} \times (\bX\rd \bG)^n\,, 
$$
which records the moduli of $(C,p_1,\dots,p_n)$ and the values
$$
(f(p_1),\dots,f(p_n)) \in (\bX\rd \bG)^n \,. 
$$
The partition function, or the index, defined by 
\begin{equation}
Z_{g,n} = \pi_* \left(z^{\deg f} \tO_\vir\right) \in K_{\textup{eq}}(\Mbar_{g,n}
\times (\bX\rd \bG)^n) [[z]] \,, \label{Zgn}
\end{equation}
gives a K-theoretic analog of a CohFT and is an object
of interest significant interest for both theoretical physicists and
enumerative geometers. The $z$-expansion in \eqref{Zgn} is about the point
\eqref{0L}. 

For instance, if $\bX\rd \bG$ is the moduli of
coherent sheaves on a surface $S$, which by the work of Nakajima \cite{Nak1}
happens when $S$ is an ADE surface, then \eqref{Zgn} gives 
equivariant Donaldson-Thomas counts of sheaves on threefolds that
fiber in $S$ over curves, see \cite{Takagi} for an introduction. 

\subsubsection{} 

One can introduce another set of marked points $\{p'_1,p'_2, \dots\}
\subset C$, which are not required to be distinct or nonsingular for
$f$. At such points $f$ evaluates to a point in the quotient stack.
Therefore, these marked points are associated with classes in
$K_{\bGt}(\bY)$.

\subsubsection{}

The following reformulation of the above enumerative setup is
motivated by both the technical advantages which it offers and from the point of
view of the
origin of the problem in supersymmetric gauge theories.

The moment map
\eqref{mu} extends to a $\bGt$-equivariant section of the
following sheaf
\begin{equation}
\bmu: \cO_{\Maps(C \to [\bX/\bG])} \to  H^0(C, \fgb^\vee_\cP) \,, \quad  \fgb^\vee_\cP = \cP
\times_{\bG}\fgb \,, 
\label{mumap}
\end{equation}
over the moduli of maps to $ [\bX/\bG]$. Here $H^0(\dots)$ means the
push-forward along the universal curve in the situation when the
source of the map varies in moduli. The quasimaps to $\bY$
are cut out by the equations $\bmu=0$. 

Instead of setting $\bmu$ to zero, one may introduce dual
variables $\xi$ taking values in the bundle $\fgb_\cP\otimes \cK_C$,
where $\cK_C$ is the canonical line bundle of $C$ with its natural
action of $\Aut(C)$. In
gauge theories, these are the superpartners of the gauge field,
twisted by $\cK_C$. In mathematics, these are called $p$-fields, see \cite{mixp}
for a comprehensive survey. We denote by
\begin{align}
  \Maps_\xi(C\to [\bX/\bG]) = \Big\{  &\textup{a principal $\bG$-bundle
                                        $\cP$ over $C$,} \notag \\
                                      &\textup{a section $f: C \to
                                        \cP\times _\bG \bX$, and} \notag\\
                                        &\textup{a section $\xi: C \to
                                         \fgb_\cP \otimes \cK_C$} \Big\}
                                          \Big/ \cong\label{Mapsxi}
\end{align}
the corresponding moduli
space. By definition, its subset $\QM_\xi\subset \Maps_\xi$ is formed
by maps that evaluate to a stable point outside a finite set of base
points. 

Serre duality 
$$
H^0(C, \fgb^\vee_\cP) \otimes H^1(C, \fgb_\cP \otimes \cK_C) \to H^1(C, \cK_C)=
\C 
$$
gives a function on the obstruction space of $\xi$, linear in $\xi$,
with the critical locus $\bmu=0$. This puts the problem into the
framework of cosection-localized virtual classes \cite{cosec} or matrix
factorizations \cite{matrix}. Operationally, the obstructions for $\xi$
replace the Euler class for the equations $\mu=0$, up to a
sign. Indeed, since the $\aroof$-genus
$$
\aroof(x)=x^{1/2} - x^{-1/2}
$$
changes sign under $x\mapsto x^{-1}$ and since $c_1(\fgb_\cP)=0$, we have
\begin{equation}
\aroof \left(\Hd (C, \fgb^\vee_\cP) \right) = (-1)^{\dim \fg} \, \aroof
\left(- \Hd (C, \fgb_\cP \otimes \cK_C ) \right) \,. \label{changearoof}
\end{equation}
Note we have already encountered the $p$-fields $\xi$ in the form of
the manifold \eqref{def_tX}. The new aspect is the twist by $\cK_C$
which they require in a global geometry.

\subsubsection{} 

The key to understanding \eqref{Zgn} is to study it for the very
simple geometry 
$$
(C,p_1)=(\bP^1,\infty)  \,, 
$$
in which, moreover, we consider the domain fixed and
\emph{do not} allow nodal degenerations. In other
words, we consider the open set 
$$
\QM_{\circ} \subset \QM(\bP^1)
$$
formed by quasimaps nonsingular at $p_1 = \infty$. While the map
$$
\ev_{\infty}: \QM_\circ \to \bX\rd\bG
$$
is not proper for quasimaps of fixed degree, the push-forward may
still be defined using equivariant localization with respect to
$$
q \in \Ct \subset \Aut(\bP^1, \infty) \,.
$$
The resulting object 
\begin{equation}
\Vertex = \ev_{\infty,*} \left(z^{\deg f} \tO_\vir\right)
\in K_{\bGa \times \Ct_q}(\bX\rd \bG)_{\textup{localized}} [[z]]
\,, \label{Vertex}
\end{equation}
is known under different names such as the disc partition function, or
the K-theoretic I-function. We call it the \emph{vertex} function to emphasize its
connection to both the vertices in DT theory, see e.g. \cite{Takagi} for an
introduction, and its connection with $q$-deformed vertex operators
and $q$-deformed conformal blocks, see \cite{AFO}. We avoid using the
term disc partition function because this would certainly cause
confusion with the maps considered in Section \ref{s_Dd} below.

\subsubsection{}

More generally, one may consider two evaluation maps
\begin{equation}
  \label{eq:18}
  \xymatrix{
    & \QM_\circ \ar[ld]_{\ev_0} \ar[rd]^{\ev_\infty} \\
    [\bX/\bG] && \bX\rd\bG
   } 
\end{equation}
and the operator $\VwD$ defined in \eqref{defbJ}. 
The vertex \eqref{Vertex} may be
interpreted as the vertex with a trivial descendent
\begin{equation}
1 \in K_{\bGt\times \Ct_q}(\bX)\,, \label{eq1}
\end{equation}
that is, as the result of acting by transposed operator on the vector
\eqref{eq1}.

As we will see, the elliptic function nature of elliptic stable
envelopes will bridge the difference between the two evaluation maps
in \eqref{eq:18}.


\subsection{Maps from $\Dd$ and $q$-Gamma functions}\label{s_Dd}

\subsubsection{}

Consider the formal disc 
$$
\Dd = \Spec \C[[t]] \,. 
$$
For any scheme $S$, one
defines 
\begin{equation}
\Maps(\Dd, S) = \lim_{\longleftarrow} \Jet_d(S), \qquad \Jet_d(S)=\Maps(\Spec  \C[[t]]/t^{d+1},
S)\,, \label{eq:26}
\end{equation}
which is a scheme, typically not of finite type over $\C$, see e.g.\ 
\cite{Mus} for an introduction.

Note that the induced action of $q\in \Aut(\Dd)$ on
$\cO_{\Maps(\Dd, S)}$ makes it a graded algebra with finitely many
generators below any given degree, thus an object suitable for the
usual K-theoretic manipulations.

\subsubsection{}

In the same fashion, one constructs moduli space of maps to a 
quotient stack $S=\big[\tS/\bG\big]$. It contains stable quasimaps 
\begin{equation}
  \label{eq:11}
  \QM(\Dd \to \tS\rdd \bG) \subset \Maps(\Dd, S)  \,,
\end{equation}
as an open substack of maps $f(t)$ that evaluate to a stable point at the
generic point $* \in \Dd$.

\subsubsection{}

Consider the evaluation at closed point 
$$
\ev_0 : \Maps(\Dd, S) \to S 
$$
and the sheaf 
\begin{equation}
  \label{eq:10}
\ev_{0,*} \cO_{\Maps(\Dd, S)} \in K_\textup{eq}(S)[[q]] \,. 
\end{equation}
Note that modulo $q$ one is computing with constant maps, therefore 
\begin{equation}
  \label{OvirY}
  \ev_{0,*} \cO_{\Maps(\Dd, S)}  =
  \cO_{S} + O(q)\,. 
\end{equation}

\subsubsection{}
The function
\begin{equation}
  \label{defphi}
  \phi(x) = \prod_{n\ge 0} (1-q^n x) \,. 
\end{equation}
solves the $q$-difference equation
\begin{equation}
\phi(qx) = \frac{1}{1-x} \, \phi(x)\label{phi_eq}
\end{equation}
and vanishes at the points $x=1,q^{-1},q^{-2}, ... $ which form a
multiplicative analog of the sequence $\{0,-1,-2,\dots\}$. Up to conventions, this is the reciprocal of the $q$-Gamma
function.

\subsubsection{}

We define $\phi$ for K-theory classes by the rule
$$
\phi( x\oplus y ) = \phi(x) \, \phi(y) \,.
$$
Note that $\phi(-x)$ is the character of a polynomial algebra with
generators of degree $x,qx,q^2x, \dots$. In particular,
\begin{equation}
\phi(-q/x) = \textup{character of polynomial functions on }\Maps(\C_q
\to \C_x)\,, \label{phimaps}
\end{equation}
where the source and the target have the indicated weights.

\subsubsection{}

Generalizing \eqref{phimaps}, we have the following 

\begin{Proposition}\label{p_1} 
  If $S$ is a smooth stack, we have 
  \begin{equation}
  \label{evOvir_phi}
 \ev_{0,*} \cO_{\Maps(\Dd, S)} = \phi(-q \, T^\vee \!S) \,. 
\end{equation}
\end{Proposition}
\begin{proof}
  If $S$ is a smooth stack, computations with positive powers of $q$ 
involve quotients of vector spaces by vector spaces. By construction,
these
reduce to \eqref{phimaps}. 
\end{proof}

\subsubsection{}\label{s_mero} 
Note that the result in \eqref{evOvir_phi} belongs to the intermediate
algebra
$$
K_\textup{eq}(S) \subset  K_\textup{eq}(S)_\mero \subset  K_\textup{eq}(S)[[q]]
$$
formed by meromorphic functions on
$$
\Spec K_\textup{eq}(S) \times \{|q|<1\}
$$
with poles along the following divisors. Let $x\in \tS$ be a point and
let $\bG_x$ denote its stabilizer subgroup. The inclusion $(\bG_x,x)
\to (\bG, \tS)$ induces a map
$$
\bG_x / \textup{conjugation} \to \Spec K_\textup{eq}(S) \,. 
$$
The pullback of the polar divisor of \eqref{evOvir_phi} is
contained in the union of the divisors $w=q^k$, where $k=1,2, \dots$
and $w$ is a weight of $T_x \tS$.

\subsubsection{} \label{s_theta} 
The function 
\begin{equation}
\vth(z) = \phi(qz) \phi(z^{-1})\label{def_vartheta}
\end{equation}
has a simple zero at $z\in q^\Z$ and satisfies
\begin{equation}
\vth(q^k z) = (-1)^k q^{-k(k+1)/2} z^{-k} \vth(z) \,.\label{theta_trans_1}
\end{equation}
We fix \eqref{def_vartheta} as the pullback of the theta line bundle
on $E$ under
the canonical map \eqref{CtoE} below. We note that different normalizations of theta functions are
convenient in different contexts. For instance, in \cite{ese} the the
more symmetric choice $(z^{1/2}-z^{-1/2})\phi(qz) \phi(q/z)$ is
used.

\subsubsection{} \label{s_theta2}
{}From \eqref{theta_trans_1}, one get the following transformation
formula for sections of $\Theta(\cV)$, where $\cV \in K_\bG(\bX)$.
Let $\bG_x\subset \bG$ denote the stabilizer of a point $x \in \bX$
and let $\bT\subset \bG_x$ denote a maximal torus. Let
$\vartheta_\cV(t)$ denote a section of the pullback of $\Theta(\cV)$
under the composed map
$$
\bT \xrightarrow{\quad \chern_{K\to E} \quad}
\Ell_{\bT}(\pt)\xrightarrow{
  \quad (\bT,\pt) \to (\bG,x) \quad} \Ell_\bG(\bX) \,, 
$$
where the first map is the reduction modulo $q$, see \eqref{chern}
below.

Recall from e.g.\ Section B.3 in \cite{part1} that the degree of a line
bundle on $\Ell_{\bT}(\pt)$ is described by an integral quadratic
form on $\cochar(\bT)$, which in this case is
\begin{equation}
(\xi,\eta)_\cV = \tr_{\cV_x} \xi \eta \,, \quad \xi, \eta \in \Lie \bG_x
\,.\label{norm_xi}
\end{equation}
This quadratic form is the same data as a symmetric linear map 
$$
 \cochar(\bT)  \owns \sigma \mapsto \sigma^\vee_\cV \in \chr(\bT) \,. 
$$
With these notation, the formula \eqref{theta_trans_1} generalizes to
\begin{equation}
\frac{\vth_\cV(\sigma(q) \, t)}{ \vth_\cV(t)   } = q^{-\|\sigma\|^2/2} \, (-q^{-1/2})^{\langle
  \sigma, \det \cV \rangle} \sigma^\vee_\cV(t)^{-1}   \,.\label{theta_trans_2}
\end{equation}

\subsubsection{}

Also note that the formula \ref{def_vartheta} relates $\vartheta$ to
the monodromy of a 
scalar $q$-difference equation closely related to \eqref{phi_eq}.
Indeed
$$
\vth(z) = f_0 / f_\infty \,, \quad f_0 = \phi(qz)\,, \quad f_\infty =
\phi(z^{-1})^{-1} 
$$
where $f_0$ and $f_\infty$ solve the equations
$$
f_0(qz) = \frac{1}{1-qz} f_0(z) \,, \quad f_\infty(qz) =
\frac{1}{1-(qz)^{-1}} f_\infty(z)
$$
and are holomorphic in the neighborhood of $0$ and $\infty$,
respectively.

\subsubsection{}

In $\Ct_q$-equivariant localization formulas, we may replace the
domain of the quasimap by the formal neighborhoods of the
$\Ct_q$-fixed points. Therefore, $q$-Gamma functions appear
in equivariant localization formulas for
exactly the reasons discussed in this section. Many authors have
observed that it is, therefore, natural to add $q$-Gamma functions
terms into the definition of vertex functions and related objects. For
instance, the full Nekrasov instanton partition function
\cite{Nekr_Instantons} includes a
double gamma function prefactor, called the perturbative
contribution. A very general framework for incorporating such factors
has been proposed in cohomology by H.~Iritani in \cite{Iri}.

Observe that the $q$-Gamma function is a much more natural K-theoretic
object than its cohomological counterpart which leads to certain transcendental constants instead of
$q$-series with integral coefficients.

\subsection{Monodromy and stable envelopes}

\subsubsection{}
Vertex functions pack information about quasimaps of all possible
degrees, which a priori is an infinite amount of
information. Remarkably, their complexity as special functions is
bounded  by certain linear $q$-difference equations that they
satisfy in both the K\"ahler variables $z$ and in equivariant
variables in $\bGa$. Further, the $q$-difference equations in
variables associated with symplectic (or more generally, self-dual)
actions are regular. This includes by our hypothesis the variables
$z$.

See, in particular, \cite{OS} for a
detailed discussion of the difference equations for quasimaps to
Nakajima quiver varieties. They involve very nontrivial constructions
in geometric representation theory.

\subsubsection{}

There are two basic questions that one can ask about a function that
solves some linear $q$-difference equation: to compute the equation
itself and/or to
compute its monodromy. Here we focus on the latter.

Recall that solutions of a regular
$q$-difference equations in $z$ are meromorphic functions on $\bZb$ and the monodromy
of the equation is the ratio of the fundamental solutions normalized
at two different starting points $0_\cL$ and $0_{\cL'}$. See \cite{Gaz,EFK} for an introductory
discussion of linear $q$-difference equations.

\subsubsection{} 

Clearly, the monodromy depends on the variables $z$ in a
$q$-periodic way, that is, it factors through the map
\begin{equation}
z \in \bZ \to \cE_\bZ=\Pic_\bG(\bY)_\tp \otimes E\,, \quad E = \Ct /
q^\Z \,. \label{defcEZ}
\end{equation}
The target of this map parametrizes the dynamical variables for
elliptic stable envelopes for $[\bY/\bG]$, and this will play an
important role in our computation of the monodromy.

\subsubsection{} 

The connection between
elliptic stable envelopes and monodromy of vertex functions has been
already noted and explored in \cite{ese}, where, in particular, the monodromy in
the equivariant variables was computed in terms of the elliptic stable
envelopes for
$$
\bA \subset \bGas \subset \bGa\,. 
$$
Here $\bGas$ denotes the subgroup of symplectic automorphisms.
This result of \cite{ese} uniquely constraints the monodromy in the K\"ahler
variable $z$ for a wide class of $\bX$, see \cite{BezOk} for an example of
this principle in action. 

Our main result
here is a direct computation of the monodromy which uses nothing but
the stable envelopes for $\bG$.

\subsubsection{}

Vertex functions belong to the world of equivariant K-theory, while
their monodromy will be expressed in terms of elliptic stable
envelopes, and thus objects from the world of equivariant elliptic cohomology.
To relate the two, we will use the canonical map
\begin{equation}
\chern_{K\to E}: \Spec K_\bG(\bX)\otimes \C \to
\Ell_\bG(\bX)\label{chern}
\end{equation}
that exists for every $\bG$ and $\bX$ and comes from the quotient map
\begin{equation}
\chern_{K\to E}: \Spec K_{U(1)}(\pt)\otimes\C  = \Ct \to \Ct/q^\Z =
\Ell_{U(1)}(\pt) \,, \label{CtoE}
\end{equation}
and the corresponding map between the formal groups of the cohomology
theories. Its role and properties are exactly parallel to the
cohomological Chern character maps like
$$
\chern_{H\to K}: \Spec \Hd_{U(1)}(\pt)\otimes\C = \C
\to \C/2\pi i \Z = \Spec K_{U(1)}(\pt)\otimes\C \,. 
$$

\subsubsection{}

Recall the canonical Poincar\'e-type line
bundle $\cU$ over $\Ell_\bG(\bX) \times \cE_\bZ$, where $\cE_\bZ$ is as in
\eqref{defcEZ}. Concretely, if $\{\cL_i\}$ is a basis of
$\Pic_\bG(\bX)_\tp$ and $\{z_i\}$ are the dual coordinates on
$\cE_\bZ$ then
\begin{equation}
\cU= \Theta(\textstyle{\sum} (\cL_i -1)(z_i-1))\,, \label{cUTh}
\end{equation}
see, in particular, Appendix A.4.1 in \cite{part1}. By the convention of Section \ref{s_theta}, sections of $\cU$ pull
back to functions of the form
\begin{equation}
\bu(s,z) = \frac{\vth(s z)} {\vth(s) \vth(z)} \label{budef}
\end{equation}
under the Chern character map \eqref{chern}. The function
\eqref{budef} satisfies 
\begin{equation}
\bu(q s,z) = z^{-1} \bu(s,z)\,, \quad \bu(s,qz) = s^{-1} \bu(s,z)
\,.\label{bueq}
\end{equation}

\subsubsection{}

Geometrically, operators of the general form \eqref{defbJ} intertwine certain interesting $q$-connections 
associated to the point $0 \in \bP^1$ with a simple abelian
$q$-connection associated to the nonsingular point $\infty \in
\bP^1$, see Chapter 8 in \cite{pcmi}.

The abelian connection contains two parts. One does not involve the
K\"ahler variables $z$ and can be traced to the edge terms in
localization formulas as in Section 8.2 of \cite{pcmi}. A flat section of
this connection may be given in terms of the $q$-Gamma functions
discussed in Section \ref{s_Dd}. We define
\begin{equation}
  \label{def_bGp}
  \bGamma'= \phi\left(-q( T \bX - \fg + \fgb)^\vee\right) 
\end{equation}
where $\fgb$ denotes $\fg$ with the $\bGa$-module structure required
in \eqref{mu}. One recognizes in \eqref{def_bGp} the tangent bundle to
$\tX/\bG$. We also consider a slightly different function
\begin{equation}
  \label{def_bG}
  \bGamma= \phi\left(-q\, T^\vee \bX + q\, \fg^\vee - \fgb^\vee\right)
  \,. 
\end{equation}
The difference between the two is that \eqref{def_bG} contains the
$p$-field term $\fgb^\vee$ without the factor of $q$. This means that
in \eqref{def_bG}, the $p$-fields are not constrained to vanish at $t=0$. 



\subsubsection{}
The other part of the abelian connection describes the $z$-dependence
and is a slight modification of the equation 
\begin{equation}
\bu(q^{\cL} z) = \cL^{-1} \otimes \bu(z) \label{eqbu}
\end{equation}
satisfied by the function \eqref{cUTh}. We recall that line bundles
$\cL$ give cocharacters of the torus $\bZ$ of the K\"ahler variables. The shift
by $q^{\cL}$ is the shift by the value of the corresponding
cocharacter at $q\in \Ct$.

To match with curve counting formulas, we will need to replace
$\cL^{-1}$ by $\cL$ in \eqref{eqbu}, and we will also need to account
for the term that comes from the degree of the
polarization. These changes will appear in the formula \eqref{cSs}
below.

It is a very general phenomenon, analogous to the shift by
$\varrho=\frac12 \sum_{\alpha >0} \alpha$ in Lie theory, that
curve-counting formulas experience a shift of the K\"ahler variables
$z$ by
a power of $(-\hbar^{1/2})$, see \cite{pcmi}. For instance, the
substitution in \eqref{cSs} may be written as
$$
z_\textup{old}^{-1} = z_\#\,, 
$$
where $z_\#$ is the variable used in Theorem 8.3.18 in \cite{pcmi}.
In cohomology, this shift  becomes
the shift by the canonical theta-characteristic, see \cite{MO1}.

\subsubsection{}

Now we are ready to state the main result of this section that
relates the operator \eqref{defbJ} with the elliptic stable envelopes.
Recall the sheaf $\cS_\symr $ from \eqref{cSG3} and denote 
\begin{equation}
\cS_\# = \cS_\symr \Big|_{\textstyle{z_\textup{old} = 
    z^{-1} (-\hbar^{1/2})^{\det T^{1/2}}}} \,.\label{cSs}
\end{equation}
The shift in \eqref{cSs}
is by the value of the cocharacter corresponding to $\det T^{1/2} \in
\Pic(X)$ at the point $(-\hbar^{1/2})\in \Ct$. Here the choice of the square root
$\hbar^{1/2}$ is fixed in the definition of the  symmetrized virtual structure
sheaf $\tO_\vir$. 
Similarly, we denote by $\Stab_\#$ the stable envelope \eqref{Stabmmu},
with the K\"ahler variables $z$ changed as above. 

In Section \ref{s_0L}, we introduced the toric variety $\bZb$ that
compactifies the K\"ahler torus $\bZ$. Let $0_\Lamp \in \bZ$ be the
torus-fixed point corresponding to an ample line bundle $\Lamp$. Let
$\bX_\sst \subset \bX$ be the corresponding stable locus. 

\begin{Theorem}\label{t3} 
  If $\gamma$ is a section of $\cS_\#$ over $\Ell_{\bGa}(\bX_\sst)$ then
  \begin{equation}
  \VwD\left(\chern_{K\to E} ^*(\gamma) \otimes \bGamma\right) =
   \hbar^{-\frac14 \dim \bX\rd\bG}
  \chern_{K\to E}^* 
  \left(\Stab_\#(\gamma)\right) \otimes \bGamma' 
\label{VwDStab}
\end{equation}%
  for $z$ in a neighborhood of $0_{\Lamp}$. 
\end{Theorem}

\subsubsection{}\label{s_c_Monodr} 
The left-hand side of \eqref{VwDStab} involves the summation over all
degrees weighted by  $z^{\deg}$. This summation converges for $z$ in a
neighborhood of $0_{\Lamp}$. Clearly, the right-hand side of the formula 
\eqref{VwDStab} gives a meromorphic continuation of its left-hand
side.  We thus conclude the following: 

\begin{Corollary}\label{c_Monodr} 
Let $\cL'$ be another choice of the stability parameter. The monodromy
of the $q$-difference equation satisfied by the function \eqref{VwDStab}
from the point $z=0_\cL$ to the point $z=0_{\cL'}$ is given by the
linear
operator
$$
\textup{Monodromy} = \left(\Stab_\#'\right)^{-1} \circ \Stab_\#\,.
$$
\end{Corollary}

\subsubsection{}\label{s_rho} 

To make \eqref{VwDStab} a scalar-valued function of the equivariant variables in $\bGa$
and the K\"ahler variables in $\bZ$, one pairs it with a test element
$\rho \in K_{\bGt}(\bX)$ using the paring 
\begin{equation}
\langle \, \cdot \, , \, \cdot \, \rangle: \left(K_\bGt(\bX) \otimes K_\bGt(\bX)_\mero  \right)_{\supp} \xrightarrow{\,\,\,
  \quad }
 K_\bGt(\pt)_\mero \xrightarrow{\,\,\, \int_\bG \,\,\,}
 K_\bGa(\pt)_\mero\,,\label{pair_int}
 \end{equation}
 where
 \begin{enumerate}
 \item[(1)] the meromorphic completion $K_\bGt(\bX)_\mero$ is as
   Section \ref{s_mero};
\item[(2)] the subscript support means that we need to make sure the set
 $$
 \supp \rho \cap \supp \textup{\eqref{VwDStab}} \cap \tX^g 
 $$
 is proper for a generic $g\in \bGt$;  (Note that for e.g.\ Nakajima quiver varieties
 $\tX^g$ is by itself proper.)
 \item[(3)] the second push-forward in \eqref{pair_int} is the projection to $\bG$-invariants. It 
may be computed by
integration over a compact form of $\bG$, whence the notation. Note
that, by definition, elements of $K_\bGt(\pt)_\mero$ are regular on
any compacts form of $\bG$. 
\end{enumerate}

\subsubsection{}\label{s_integral}

The Weyl integration formula  expresses $\int_\bG$ as an integral in
which the measure is the Haar measure on a compact torus and the
integrand is a meromorphic function on the corresponding complex
torus. Moreover, the $q$-Gamma functions in \eqref{VwDStab} make this
integral a relative of the Mellin-Barnes integrals popular in the
classical theory of hypergeometric functions.

Such integrals may be computed by residues, which is a standard
practice in the 
physics literature going back to at least \cite{Higgs}, in a closely
related context. See, in particular, the discussion in
\cite{Hori_Beij, Hori_Tong} and also e.g.\ the Appendix in \cite{AFO}
for a discussion aimed at mathematicians. 

Our proof of Theorem \ref{t3} may, in principle, be recast in a form
that refers to contour deformation and residue computations. For
the convenience of those readers who find this language more
familiar, we provide an example of such
translation in Appendix \ref{s_A}.

\subsection{Proof of Theorem \ref{t3}}

\subsubsection{}

Consider the moduli spaces $\bM_i$, $i=1,\dots,3$, and the sheaves
$\tO_1,\cO_2,\cO_3$ described in the following table:

 \begin{center} 
  \begin{tabular}{ | m{6.7cm}| m{6cm} | } 
\hline
 stable quasimaps \vspace{-0.3cm} $$\bM_1 = \QM_{\xi,\circ}(\bP^1 \to
    [\bX/\bG]) \vspace{-0.3cm} $$ 
    with $p$-fields $\xi$, nonsingular at
                 $\infty \in \bP^1$& $\tO_1$ = symmetrized virtual
                                     structure sheaf
                                     $\tO_{\bM_1,\vir}$ of $\bM_1$, cosection
                                     localized to quasimaps to $\bY$\\ 
  \hline
  stable quasimaps \vspace{-0.3cm}$$\bM_2=\QM_{\xi}(\Dd \to
                   [\bX/\bG]) \vspace{-0.3cm}$$ with $p$-fields $\xi$
                                   & $\cO_2$ = virtual
                                     structure sheaf of $\bM_2$, cosection
                                     localized to quasimaps to $\bY$ \\
  \hline
 all maps \vspace{-0.3cm}$$\bM_3=\Maps_{\xi}(\Dd\to
                   [\bX/\bG]) \vspace{-0.3cm}$$ with $p$-fields $\xi$
                                   & $\cO_3=$ virtual
                                     structure sheaf of $\bM_3$\\
  \hline
\end{tabular}
\end{center}

\noindent 
By definition, a map from the formal disc $\Dd$ is stable if it
evaluates to a stable point at the generic point of $\Dd$. 

\subsubsection{}

By its definition \eqref{defbJ}, the operator $\VwD$ is given by the
formula
\begin{equation}
  \label{vvv1}
  \VwD = \ev_{0,*} \left( z^{\deg f} \tO_1 \otimes
  \ev_\infty^*(\, \cdot \,) \right) \,. 
\end{equation}
Our strategy is to compare this integration over $\bM_1$ with the integration over
$\bM_2$ and $\bM_3$. We begin with the following result that connects
$\bM_1$ with $\bM_2$. 

\begin{Proposition}\label{p_vv1} 
  For $\gamma$ as in Theorem \ref{t3}, we have
  \begin{equation}
    \label{vvv2}
    \VwD\left(\chern_{K\to E} ^*(\gamma) \otimes \bGamma\right) =
    \hbar^{-\frac14 \dim \bX\rd\bG}
  \chern_{K\to E}^* 
  \left(\Stab_\#(\gamma)\right)  \otimes \ev_{0,*} \cO_2 \,, 
  \end{equation}
  for $z$ in a neighborhood of $0_\Lamp$. 
\end{Proposition}

\noindent
The proof will occupy Sections \ref{s_pvv1_b}---\ref{s_pvv1_c}.

\subsubsection{Proof of Proposition \ref{p_vv1}} \label{s_pvv1_b}

Recall that, by construction, the pushforward in \eqref{vvv1} is
computed in $\Ct_q$-equivariant K-theory. Let us consider the corresponding
fix loci and their contributions. Let
$$
(\cP, f, \xi) \in \bM_1^{\Ct_q}
$$
be a $q$-fixed quasimap. This means that there is an action of $\Ct_q$ on $\cP$ such that 
$$
q \tf = \tf q \,, \quad \tf = (f,\xi) \,, 
$$
for the induced action of $\Ct_q$ on $\cP \times_\bG \bX$ and
$\fgb_\cP \otimes \cK_C$. Note that the action of $q\in \Aut(\bP^1)$ 
lifts to a \emph{unique} automorphism of the bundle $\cP$ because
$\tf(t)$ is stable for $t\ne 0,\infty$ and hence has trivial stabilizer in
$\bG$. 

\subsubsection{}
We cover $\bP^1$ by two charts $\A^1$ and $\bP^1\setminus \{0\}$ and
trivialize $\cP$ in both charts. In each chart, $\Ct_q$ acts by a
coweight of $\bG$.

Consider the chart at $\infty$ and the corresponding coweight
$$
\sigma_\infty : \Ct_q \to \bG \,.
$$
The point 
$$
\sigma_\infty(q) f(\infty) = f(q \infty) = f(\infty) \in \bX^{\sigma_\infty}
$$
is stable by the definition of $\bM_1$, and hence cannot be fixed by a
nontrivial $1$-parameter subgroup. We conclude the  following:
\begin{enumerate}
\item[(1)] the coweight $\sigma_\infty$ is trivial; 
\item[(2)] the map $f(t)=x_\infty$ takes a constant stable value in the chart near infinity; 
\item[(3)] $\xi=0$. 
\end{enumerate}
The vanishing of $\xi$ follows from the absense of $q$-invariant
sections of $\cK_{\bP^1\setminus 0}$.

\subsubsection{}
Now consider the chart of at zero. Since $\sigma_\infty =1$, the
corresponding coweight is the same as the clutching function for
the bundle $\cP$. We denote it by simply $\sigma$. We observe that the
point $x_\infty$ is such that
$$
x_0 = \lim_{q\to 0} \sigma(q) \cdot x_\infty
$$
exists, is fixed by $\sigma$,  and therefore is unstable unless $\sigma$ is trivial. If $\sigma$ is
trivial then the quasimap is constant and $x_0=x_\infty$.

We conclude that
\begin{equation}
  \label{xinf}
  \bM_1^{\Ct_q} = \bigsqcup_{\sigma/ \textup{conjugacy}} \left\{(x_0,x_\infty) \big| \, x_0 \in
  \bX^\sigma, x_\infty\in \Attr_\sigma(x_0) \cap \bX_\sst \right\}
\end{equation}
and the evaluation map
\begin{equation}
  \label{evev}
    \bM_1^{\Ct_q} \xrightarrow{\quad (\ev_0,\ev_\infty)\quad} \bX
    \times \bX 
\end{equation}
is the map to $(x_0,x_\infty)$.


\subsubsection{}\label{s_tgam} 
We denote
\begin{equation}
\tgam= \chern_{K\to E}^* 
  \left(\Stab_\#(\gamma)\right) \,, \label{def_tgam}
  \end{equation}
and consider the pullback of this class by the map
\begin{equation}
\sigma\times \iota: (\Ct_q,  \bX^\sigma ) \to (\bG,  \bX )
\,. \label{sigmaid}
\end{equation}
The elliptic transformation property of the class $\tgam$
allows us to compare its pullback via different coweights $\sigma$
in \eqref{sigmaid}. From definitions, see in particular Section
\ref{s_theta}, we compute
\begin{equation}
  \label{eq:14}
  (\sigma \times \iota )^* \, \tgam = c_\sigma \, z^{\deg(\sigma)} (1 \times \iota
  )^* \,  \tgam \,, 
\end{equation}
where $\deg(\sigma)$ is the image of $\sigma$ under the map
\begin{equation}
\deg(\sigma): \textup{coweights}(\bG) \to \Hom\left(\Pic_\bG(\bX),\Z\right) \to
\textup{characters}(\bZ) \,.  \label{degsigma}
\end{equation}
The other factor in \eqref{eq:14} is given by 
\begin{equation}
  \label{csigma}
  c_\sigma^{-1} =  q^{\beta_2}\left(\hbar^{1/2}
    q^{1/2} \right)^{\beta_1}  \, \sigma^\vee_{T^{1/2}} 
    \end{equation}
with
\begin{equation}
  \label{eq:17}
   \beta_1= \langle \sigma, \det T^{1/2} 
              \rangle  \,,\quad 
              \beta_2 = \tfrac12 (\sigma,\sigma)_{T^{1/2}}   
             \,. 
\end{equation}
In other words, if $f^*(T^{1/2})$ splits as $\bigoplus
\cO_{\bP^1}(m_i)$, which is equivalent to $\sum q^{m_1}$ being the character of
$\ev_0^*(T^{1/2})$, then 
$$
\beta_1 = \sum m_i\,, \quad \beta_2 = \frac12 \sum m_i^2 \,.
$$
The $c_\sigma \, z^{\deg(\sigma)}$ factor in \eqref{eq:14} is the
automorphy factor for sections of the line bunde $\cS_\#$ in
\eqref{cSs}. It combines the contributions from $\Theta(T^{1/2})$ and
$\cU$. In particular, the monomial in $z$ comes from the first
equation in \eqref{bueq}.

We remark that this is the place in the proof where
the exact form of the $z$-shift in \eqref{cSs} is used. 

\subsubsection{}\label{s_ev_ev}

Note that in \eqref{evev}, $\Ct_q$ acts via $\sigma$ on the first factor and
trivially on the second factor. Further, $\gamma$ is the
restriction of $\Stab_\#(\gamma)$ to the stable locus, and the
restriction of any class to $\Attr(\bX^\sigma)$ is the same as the
pullback of its restriction from $\bX^\sigma$. Therefore, 
from \eqref{xinf} and \eqref{eq:14} we conclude that 
\begin{equation}
  \label{eq:8}
  \ev_\infty^*(\chern_{K\to E}^*(\gamma)) =
  c_\sigma^{-1} z^{-\deg(\sigma)}  \, \ev_0^*(\tgam) \, .
\end{equation}

\subsubsection{}

We now turn to the localization of $\tO_1$, which
by definition is given by
\begin{equation}
  \label{tOvir}
\tO_\vir = \cO_\vir \otimes \left( \det \cK_\vir \,\,  \frac{\ev_\infty^*(\det
    T^{1/2})}{\ev_0^*(\det T^{1/2})}
  \right)^{1/2} \,, 
  \end{equation}
  where
  \begin{equation}
    \label{Kvir}
    \cK_\vir = \det( T_\vir \bM_1)^{-1}
  \end{equation}
  is the virtual canonical bundle. We note that 
  \begin{equation}
    \label{eq:13}
\left. \frac{\ev_\infty^*(\det
    T^{1/2})}{\ev_0^*(\det T^{1/2})} \right|_{\cM_1^{\Ct_q}} =
q^{-\beta_1}
\,, 
  \end{equation}
  by the argument of Section \ref{s_ev_ev}.

  \subsubsection{}
  We now consider $T_\vir \bM_1$, which we split in two parts 
  $$
  T_\vir \bM_1 = \Hd\left(\bP^1, V_1 + V_2 \right)
  $$
  where
  \begin{equation}
  V_1 = f^* T \bX\,, \quad V_2 = - \fg_\cP + \fgb_\cP \otimes \cK_{\bP^1}\,.
\label{V1V2}
\end{equation}
  If $\cL\cong \cO(m)$ is a line bundle on $\bP^1$ then
  $$
  \det \Hd (\bP^1, \cL)  = q^{m(m+1)/2}
  \left(\cL\big|_\infty\right)^{m+1} \,. 
  $$
Since $V_1\cong \sum \cL_i + \hbar^{-1}
\cL_i^\vee$, we conclude 
\begin{equation}
  \label{eq:15}
  \det \Hd(V_1) = q^{2 \beta_2} \, \hbar^{\beta_1+ \frac12 \dim \bX}
 \ev_\infty^*(\sigma^\vee_{T^{1/2}})^2  \,. 
\end{equation}
Similarly, $V_2\cong \sum (\hbar 
\cL_i^\vee \otimes \cK_{\bP^1} - \cL_i)$ for some $\cL_i$. Moreover,
the bundle $\sum \cL_i$ is self-dual and, in particular, of zero
degree. It follows that 
\begin{equation}
  \label{eq:19}
  \det \Hd(\bP^1,V_2) = \hbar^{-\dim \fg} \,. 
\end{equation}

\subsubsection{}
Putting \eqref{eq:8}, \eqref{eq:13}, \eqref{eq:15}, and \eqref{eq:19}
together, we conclude
\begin{equation}
  \label{eq:20}
  z^{\deg(\sigma)} \, \ev_\infty^*(\chern_{K\to E}^*(\gamma)) \, \tO_1 =
  \hbar^{-\frac14 \dim \bX\rd \bG} \ev_0^*(\tgam) \cO_{\bM_1,\vir} \,. 
\end{equation}

\subsubsection{}
It remains to consider the contribution of $\cO_\vir$. Let $\cL$ be a
line bundle on $\bP^1$. Consider the super vector space
$\Hd(\bP^1,\cL)$, which has the character 
$$
\Hd(\bP^1,\cL) = \frac{\cL\big|_0 -
  q^{-1}\cL\big|_\infty}{1-q^{-1}} \,, 
$$
by localization. This implies the following formula for the character
of the virtual
structure sheaf of this super space
$$
\cO_{\Hd(\bP^1,\cL), \vir} =
{\phi\left(q \cL^\vee\big|_\infty- \cL^\vee\big|_0\right)} \,. 
$$
To get the character of $\cO_{\bM_1,\vir}$, we apply this formula to
$\Hd(\bP^1,V_1+V_2)$, as in \eqref{V1V2}.  We get
\begin{align}
  \cO_{\bM_1,\vir}  &= \phi\big(
    \ev_\infty^* (\textup{term at $\infty$}) - \ev_0^* (\textup{term
                      at $0$}) \big) \,, \notag \\
  \textup{term
  at $0$} &= T^\vee \bX - \fg^\vee + q \fgb \,,    \label{eq:24}
 \\
  \textup{term
  at $\infty$} &= q T^\vee \bX - q \fg^\vee + \fgb^\vee \,, \notag 
\end{align}
where the terms with trivial $q$-weight are interpreted as deformation
along the fixed manifold. 

  \subsubsection{} \label{s_pvv1_c}

  {}From \eqref{eq:24}, we observe that
  \begin{equation}
    \label{eq:25}
     \cO_{\bM_1,\vir}  \otimes \ev_\infty^*(\bGamma) = \phi\big(
   - \ev_0^* (\textup{term
                      at $0$}) \big) = \cO_2  \,. 
  \end{equation}
  Observe that the map 
  \begin{align}
    \bM_1^{\Ct_q} &\to  \bM_2^{\Ct_q} \,, \notag \\
    (x_0,x_\infty) & \mapsto f(t) = \sigma(t) \cdot x_\infty  \label{eq:21}
  \end{align}
  is a closed embedding and its image it the locus where the $p$-field
  $\xi$ vanish.
  Note that on the formal
  disk $\Dd$ it
  is possible to have nonvanishing $q$-invariant sections of $\fgb_\cP
  \otimes \cK_{\Dd}$ for a nontrivial coweight $\sigma$. However,
  since $\cO_2$ is cosection localized to maps to $\bY$, the support
  of its localization is the image of \eqref{eq:21}. We conclude that
  \begin{equation}
    \label{eq:22}
    \ev_{0,*} \left( z^{\deg} \, \tO_1 \otimes
      \ev_\infty^*(\chern_{K\to E}^*(\gamma)\otimes \bGamma) \right) =
     \hbar^{-\frac14 \dim \bX\rd\bG} \tgam \otimes \ev_{0,*} \cO_2 \,, 
   \end{equation}
   and this completes the proof of Proposition \ref{p_vv1}. \hfill $\square$

  \subsubsection{}

The second half of the proof of Theorem \ref{t3} is summarized in the
following

\begin{Proposition}\label{p_vv2} 
  For $\gamma$ as in Theorem \ref{t3}, we have
  \begin{equation}
    \label{vvv3}
  \chern_{K\to E}^* 
  \left(\Stab_\#(\gamma)\right)  \otimes \ev_{0,*} \cO_2  =
  \chern_{K\to E}^* 
  \left(\Stab_\#(\gamma)\right)  \otimes \ev_{0,*} \cO_3 \,, 
  \end{equation}
  for $z$ in a neighborhood of $0_\Lamp$. 
\end{Proposition}

\noindent
The proof will occupy Sections \ref{s_pvv2_b}---\ref{s_pvv2_c}.

\subsubsection{Proof of Proposition \ref{p_vv2}} \label{s_pvv2_b}

Recall that $\cO_3$ is the virtual structure sheaf of the larger of
the two spaces 
$$
\bM_3 = \Maps_\xi(\Dd \to [\bX/\bG]) \supset \bM_2. 
$$
The difference between $\cO_3$ and $\cO_2$ is twofold:
\begin{enumerate}
\item[(1)] $\cO_3$ is not cosection localized to maps to $\bY$,
\item[(2)] $\cO_3$ is not restricted to maps that are stable at the generic
  point $*\in \Dd$. 
\end{enumerate}
We address the first point first.

\subsubsection{}\label{s_wheel} 

Consider the target of the map \eqref{eq:21}
\begin{equation}
  \label{xinf2}
  \bM_2^{\Ct_q} = \bigsqcup_{\sigma/ \textup{conjugacy}} \left\{(\txx_0,\txx_\infty) \big| \, \txx_0 \in
    \tX^\sigma, \txx_\infty\in \Attr_\sigma(\txx_0) \cap \tX_\sst
  \right\}\,, 
\end{equation}
where $\txx = (x,\xi)$. Note that, while $\xi$ may be nonvanishing on
$\bM_2^{\Ct_q}$, the function
\begin{equation}
  W(t)=\langle \xi(t), \mu(x(t)) \rangle = H^0(\Dd,\cK_\Dd)
\end{equation}
has to vanish, as there are no $q$-invariant sections of
$\cK_\Dd$.

As pointed out in Section \ref{s_rmks}
$$
\ev_0^*(\tgam) = \chern_{K\to E}(s'') \big|_{\Attr(\bX^\sigma)} \,, 
$$
where $s''$ is the elliptic class on $\tX$ constructed in Section \ref{s_spp} in the
course of the proof of Theorem \ref{t2}.  By construction,
$$
s'' \Big|_{(\bX_\sst \cap W^{-1}(0)) \setminus \bY} = 0 \,, 
$$
because $s''\big|_{\bX_\sst}$
is supported on the descending manifold of $\bY$ with respect to
the gradient flow of $\Re W$. This means that $\ev_0^*(\tgam)$ is
already supported on quasimaps to $\bY$ and no cosection localization
of $\cO_2$ is required. 

\subsubsection{}

By definition \eqref{eq:26}, there are canonical maps
\begin{equation}
  \label{eq:23}
  \jet_d: \bM_3 \to \Jet_{d}\,, \quad \Jet_d
  =\Jet_{\xi,d}([\bX/\bG]) \,. 
\end{equation}
We denote by $\Jet_{d,\ust}\subset \Jet_d$ jets with values in the
unstable locus and define 
$$
\Jet_{d,\sst} = \Jet_d \setminus \Jet_{d,\ust} \,.
$$
This is the stable locus for the action of $\bG$ on jets of maps
to $\bX$. Clearly,
\begin{equation}
\bM_3 \setminus \bM_2  = \lim_{\longleftarrow}  \Jet_{d,\ust} \,
.\label{eq:27}
\end{equation}

\subsubsection{}\label{s_norm_rho}
To compare the pushforwards from $\bM_3$ and $\bM_2$, we consider the
sheaf $\jet_{d,*} \cO_3$ and consider its decomposition
$$
\jet_{d,*} \cO_3 = \jet_{d,*} \cO_3 \big|_{\Jet_{d,\sst}} +
\cR_d
$$
that
corresponds
to the partition of $\Jet_d$ into the stable and unstable locus in
suitably
completed $K_\bGt(\Jet_d)$ as in \cite{DHLloc}.
The remainder sheaf $\cR_d$ may be further broken down in pieces
$$
\cR_d = \sum \cR_{i,d}
$$
that
correspond to various strata of the unstable locus as in \eqref{XiX}. 

We claim that the remainder $\cR_d$ satisfies a bound of the
form 
\begin{equation}
  \label{eq:28}
\chern_{K\to E}^* 
\left(\Stab_\#(\gamma)\right)  \otimes \ev_{0,*} \cR_d  =
O( (\const_1 z)^{\const_2 d}) \,, \quad z \to 0_\Lamp\,, 
\end{equation}
for certain positive numbers, of which the 
first one depends on other variables in the theory. The bound should be
understood via pairing with test elements $\rho \in K_{\bGt}(\bX)$ as
in Section \ref{s_rho}.
Any norm of this pairing will be bounded by $O( (\const_1 z)^{\const_2
  d})$
times a suitable norm of $\rho$.

It is clear that a bound of the form \eqref{eq:28} implies Proposition
\ref{p_vv2}, so it remains to prove this bound.

\subsubsection{}

Consider the filtration of the form \eqref{XiX}
\begin{equation}
\tX_\ust=\tX_1 \supset \tX_2 \supset \tX_3 \dots\label{tXiX}
\end{equation}
Denote by
$\Jet_{i,d}$ those jets that take values in the closure of $\tX_i$, but not in any
further pieces of the filtration \eqref{tXiX}. Let $\bM_{3,i,d}$ denote the
preimage of $\Jet_{i,d}$ under the map \eqref{eq:23}.

\subsubsection{}\label{s_conj_g} 

Let $\bj(t)\in \Jet_{i,d}$ be a jet of a map. By definition, this is
the vertical arrow in the following diagram:
\begin{equation}
  \label{eq:30}
  \xymatrix{
    & \Spec \C[[t]]/t^{d+1} \ar[d]^{\bj} \ar[ld]\\
   \bG \times_{\sP_i}\Attr_{\sigma_i}\!\big(\tX^{\sigma_i}\big) \ar[r]&
   \textup{closure}\big(\tX_i\big)
   } \,. 
\end{equation}
The horizontal arrow in the diagram \eqref{eq:30} is the action map as
in  Section \ref{s_Xi_induct}. We recall, see e.g.\ Theorem 5.6 in
\cite{VinPop}, that the horizontal map in \eqref{eq:30} is a resolution of
singularities.

Since $\bj(t)$ meets $\tX_i$, it has a unique lift to the diagonal map in
\eqref{eq:30} for $d \gg 0$. This means there is an element
$$
\bg(t) \in \Maps(\Spec \C[[t]]/t^{d+1}, \bG) \,, 
$$
unique up to the left action of maps to $\sP_i$, such that
$$
\bg(t)^{-1} \bj(t) \in \Maps\big(\Spec \C[[t]]/t^{d+1},
\Attr_{\sigma_i}\!\big(\tX^{\sigma_i}\big)\big)\,.
$$

\subsubsection{}

By construction, the moduli space $\bM_3$ involves a
quotient by the group $\Maps(\Dd,\bG)$. By the argument of Section
\ref{s_conj_g},  we may represent elements of $\bM_{3,i,d}$ by
maps $\tf(t)$ such that their $d$-jets 
\begin{enumerate}
\item[(1)] take values in $\Attr_{\sigma_i}\!\big(\tX^{\sigma_i}\big)$, and 
\item[(2)] meet $\tX_i$. 
\end{enumerate}
This means that the map $\sigma_i(t)^k \tf(t)$ is regular for $k <
\ell(d)$, where $\ell(d)$ is linear function of $d$ with positive
slope. Therefore, the map
$$
\tau_{i,d}: \bM_{3,i,d} \owns \tf(t) \mapsto \sigma_i(t)^{\lfloor \ell(d)/2 \rfloor} \tf(t) \in
\bM'_{3,i,d} \subset \bM_3 
$$
is an isomorphism with a certain subset $\bM'_{3,i,d} \subset \bM_3$
that may be described in terms of tangency of its jets with
$\tX^{\sigma_i}$. The order of this tangency grows linearly with $d$
and for $d\to\infty$ we get maps to $\tX^{\sigma_i}$.

\subsubsection{} 

We now compare
\begin{equation}
\chi(\bM_{3,i,d}, \ev_0^*(\tgam \otimes \rho) \otimes \cR_{i,d})  =
\chi(\bM'_{3,i,d}, (\tau_{i,d}^{-1})^* \left(\ev_0^*(\tgam \otimes
  \rho) \otimes \cR_{i,d}\right))  \label{eq:31}
\end{equation}
with
\begin{equation}
\chi(\bM'_{3,i,d}, \ev_0^*(\tgam \otimes
  \rho) \otimes \cR'_{i,d}))\label{eq:32}
\end{equation}
where $\cR'_{i,d}$ is the analog of $\cR_{i,d}$ for $\bM'_{3,i,d}$ and
$\rho \in K_{\bGt}(\tX)$ is a fixed compactly supported test insertion as in Section
\ref{s_norm_rho}.

Concretely, \eqref{eq:32} may be computed e.g.\
using equivariant localization for the $\Ct_q$-action and it is given
by the same sum as the computation of $\chi(\bM_3, \ev_0^*(\tgam \otimes
  \rho) \otimes \cO_3)$, except that it is restricted to fixed points
  that lie in $\bM'_{3,i,d}$. The difference between \eqref{eq:31} and
  \eqref{eq:32} in localization formulas is
  that the equivariant variables in $\bG$ are shifted by
  ${\sigma_i(q)^{\lfloor \ell(d)/2 \rfloor}}$, as we have already
  seen in Section \ref{s_tgam}.

  \subsubsection{}

  We claim that: 
\begin{enumerate}
\item[(1)] \eqref{eq:32} is finite in the $d\to \infty$ limit,
  while 
  \item[(2)] \eqref{eq:31} is satisfies
\begin{equation}
\textup{\eqref{eq:31}}  = O\left(
z^{-\lfloor\ell(d)/2\rfloor  \deg(\sigma_i)} e^{\const d} \, \big|
\textup{\eqref{eq:32}} \big| \right) \label{eq:33} \,. 
\end{equation}
\end{enumerate}
The claim about finiteness of \eqref{eq:32} is
seen directly, in the same way as Proposition \ref{p_1}.

By the construction of the stratification \eqref{tXiX}, the
cocharacter $\sigma_i$ pairs negatively with any ample line bundle
$\Lamp$ in a given stability chamber. This means that
$$
- \deg(\sigma)_i  \subset \textup{effective cone for $\Lamp$}\,, 
$$
and therefore
$$
z^{-\deg(\sigma_i)} \to 0 \,, \quad z \to 0_\Lamp \,.
$$
This means that the estimate \eqref{eq:33} proves the
estimate \eqref{eq:28}. It thus remains to establish the estimate \eqref{eq:33}. 

\subsubsection{}\label{s_pvv2_c}

The $z$-dependent factor in \eqref{eq:33} comes from the elliptic
transformation for $(\tau_{i,d}^{-1})^* \left(\ev_0^*(\tgam)\right)$
as in \eqref{eq:14}. It thus remains to show that the rest of the
terms grow at most exponentially $d$, with some fixed exponent.

The test insertion $\rho$ depends polynomially on the equivariant
variables, and therefore grows at most exponentially under $q$-shifts.

Terms that could potentially have a superpolynomial growths come from
the line bundle $\Theta(T^{1/2})$ and the $q$-Gamma functions that
appear in the localization of $\cO_3$. However, they precisely balance
out, as the following computation shows. Let $x$ be a Chern root of
$T^{1/2}$ and of $T\bX$. Since $T\bX$ is self-dual up to the action
of $\bGa$, let $(\hbar x)^{-1}$ denote the dual Chern root of
$T\bX$. The function 
$$
\frac{\vth(x)}{\phi(1/x) \phi(\hbar x)} = \frac{\phi( qx)}{\phi(\hbar x)}
$$
satisfies a \emph{regular} $q$-difference equation in the variable
$x$, meaning that its $q$-shits grow at most exponentially.

This concludes the proof of Proposition \ref{p_vv2} and Theorem
\ref{t3}.

\appendix

\section{Appendix} 

\subsection{Mellin-Barnes integrals and vertex functions}\label{s_A} 

\subsubsection{} 

The goal of this Appendix is to illustrate our proof of Theorem
\ref{t3} with a translation of the argument in the language of the
contour deformation and residue computations, as promised in
Section \ref{s_integral}.

This will be done in the most basic example of the theory, for the
simplest possible Nakajima quiver variety, namely 
$$
\bY= \textup{cotangent bundle of
  the Grassmann variety $\Gr(k,n)$} \,.
$$
This comes from
\begin{align*}
  \bX &= \Hom(W, V) \oplus \hbar^{-1} \otimes \Hom(V, W) \,, \\
  \bG &= GL(V)\,, \quad \bGa = GL(W) \times \Ct_\hbar \,, 
\end{align*}
where $W \cong \C^n$ and $V \cong \C^k$. This data is conveniently
visualized using the quiver data
\begin{equation}
  \label{eq134}
  \begin{tikzcd}
    W   \arrow[rr,bend left, "A"] && V \arrow[ll,bend left, "B"] \arrow[
  out=-30,
  in=30,
  loop,
  distance=1.5cm, "\xi" ']
  \end{tikzcd} \,, 
\end{equation} 
in which we have also included the $p$-field
$$
\xi \in \fgb = \hbar \otimes \End(V)
$$
for future reference. We have 
$$
\mu = A B \in \hbar^{-1} \otimes \End(V) = \fgb^\vee \,
$$
with the paring
$$
\langle  \mu, \xi \rangle = \tr AB \xi \,.
$$
\subsubsection{}

There are two possible choice of the stability parameter
$$
\cL_{\pm} = \cO_\bX \otimes {\det}^{\pm 1} \,, 
$$
where $\det$ is the determinant character for $\bG$.  The corresponding
stable loci are:
\begin{align*}
  \bX_{\sst,+} & = \textup{ map $A$ is surjective} \,, \\
  \bX_{\sst,-} & = \textup{ map $B$ is injective} \,. 
\end{align*}
For the space $\tX$, the corresponding conditions read
$$
\tX_{\sst,+} = \left\{ \sum_i \Image(\xi^i A) = V\right\}\,,
\quad
\tX_{\sst,-} = \left\{ \bigcap_i \Ker(B \xi^i) = 0\right\} \,. 
$$

\subsubsection{}

A principal $\bG$-bundle $\cP$ over $\bP^1$ is the same data as a
vector bundle $\cV$. The quiver data is then promoted to bundle maps
\begin{equation}
  \label{eq135}
  \begin{tikzcd}
    \cW =W \otimes \cO_{\bP^1}   \arrow[rr,bend left, "A"] && \cV \arrow[ll,bend left, "B"] \arrow[
  out=-30,
  in=30,
  loop,
  distance=1.5cm, "\xi" ']
  \end{tikzcd} \,, 
\end{equation} 
For an $\cL_\pm$-stable quasimap, the bundle $\cV$ (respectively, $\cV^\vee$) is
globally generated, thus
$$
\cV  \cong \sum \cO_{\bP^1}(d_i) \,, 
$$
where $\pm d_i \ge 0$ for $\cL_{\pm}$.

\subsubsection{}

We denote the tori in $\bG$ and $\bGa$ by
$$
\bT_\bG= \diag (x_1,\dots,x_k) \subset \bG \,, \quad \bA = \diag(a_1,\dots,
a_n) \subset GL(W) 
\,. 
$$
Then
$$
T\bX - \fg + \fgb = \sum \frac{x_i}{a_j} + \sum \frac {a_j}{\hbar x_i} -
\sum \frac{x_i}{x_j} + \hbar \sum \frac{x_i}{x_j}  \,, 
$$
and thus
\begin{equation}
\bGamma' = \prod_{i=1}^k \prod_{j=1}^n\frac{1}{\phi(q a_j/x_i) \, \phi(q \hbar x_i/a_j)}
\prod_{i,j \le k} \frac{\phi(q x_j/x_i)}{ \phi(q x_j/\hbar x_i)}
\,. \label{eq:34}
\end{equation}
Note that the diagonal $i=j$ terms in the second factor contribute
a factor $\left(\frac{\phi(q)}{\phi(q/\hbar)}\right)^k$, which is does not depend on
$x$.



Adding the contribution of $\cO_{[\bX/\bG]}$ to $\bGamma'$, we
introduce the function 
\begin{equation}
  \label{eq:36}
   \bPhi = \prod_{i=1}^k \prod_{j=1}^n\frac{1}{\phi(a_j/x_i) \, \phi(\hbar x_i/a_j)}
\prod_{i, j \le k} \frac{\phib(x_j/x_i)}{ \phi(q x_j/\hbar x_i)}
\,,  
\end{equation}
where
\begin{equation}
  \label{eq:48}
  \phib(x) =
  \begin{cases}
    \phi(x)\,, & x \ne 1 \\
    \phi(q) \,, & x =1 \,. 
  \end{cases}
\end{equation}
One may avoid introducing $\phib$ by taking out the constant
$\left(\frac{\phi(q)}{\phi(q/\hbar)}\right)^k$
and restricting the product to $i \ne j$ in the second part of
\eqref{eq:36}. 

\subsubsection{}

We may take the following polarization
$$
T^{1/2} \bX = \Hom(W,V) = \sum \frac{x_i}{a_j} \,.
$$
Which means that the quadratic form \eqref{norm_xi} takes the
following form
\begin{equation}
\| (\xi,\alpha) \|^2 = \sum (\xi_i - \alpha_j)^2 \,, \quad
(\xi,\alpha) \in \Lie(\bT_\bG) \oplus \Lie(\bA) \,.\label{formT12}
\end{equation}

\subsubsection{}

In what follows, we assume that 
\begin{equation}
  \label{eq:41}
  |q|<|\hbar| < |a_i| < 1, \quad \forall i \in \{1, \dots, n\} \,, 
\end{equation}
which corresponds to the convergence of 
$$
\cO_{\Maps_{\xi}(\Dd\to
  \bX)} =
\prod_{i=1}^k \prod_{j=1}^n\frac{1}{\phi(a_j/x_i) \, \phi(\hbar x_i/a_j)}
\prod_{i\ne j \le k} \frac{1}{ \phi(q x_j/\hbar x_i)}
$$
for $|x_i|=1$.  Then by the Weyl integration formula, we have
\begin{equation}
  \label{eq:37}
  \langle \rho, \tgam \otimes \bGamma' \rangle = \frac{1}{k!} \int_{|x_i|=1}
  \rho \, \tgam \, \bPhi \prod \frac{dx_i}{2\pi i x_i} \,. 
\end{equation}
Here $\rho$ and $\tgam$ are, in principle, arbitrary K-theory
classes. In our context
\begin{equation}
  \label{eq:38}
  \tgam = \textup{elliptic stable envelope for $T^*\Gr(k,n)$ }\,, 
\end{equation}
pulled back via the map
$$
\chern_{K\to E} (x_i) = x_i \bmod q^\Z \,.
$$
Theorem \ref{t3} specializes to the following: 

\begin{Proposition}\label{p_MB} 
The integral \eqref{eq:37} is the vertex with descendents, in which the
class $\gamma \otimes \bGamma$ is inserted at $\infty\in \bP^1$, and
the class $\rho$ inserted at $0\in \bP^1$. 
\end{Proposition}

In this Appendix, we will retrace the main steps of the proof of
Theorem \ref{t3} in the current example and in the language of
Mellin-Barnes integrals. 

\subsubsection{}

Explicit formulas for stable envelopes of the $\bA$-fixed points in
$\bY$ may be taken from  Section 4.4.5 of  \cite{ese}. There, they 
are given 
as rational functions in the Chern roots $x_i$ and interpreted as
elliptic cohomology classes on $\bY_\sst$. To make them agree
with our definition, one nees to:
\begin{enumerate}
\item[(1)] Multiply by $$
\Theta(\fgb) = \prod_{i,j} \vartheta(\hbar
x_i/x_j)\,, 
$$
which corresponds to the pushforward from $\bY_\sst$ to
$\bX_\sst$. One notes that since the resulting expression is regular,
it defines an elliptic cohomology class on all of $\bX$.
\item[(2)] Change the variable $z$ by $z_\textup{old} =
  (-\hbar^{-1/2})^n z^{-1}$ as in \eqref{cSs}. 
\end{enumerate}
The 
result if the following.

\subsubsection{}

We introduce
\begin{equation}
  \label{eq:39}
  \bff_m(x,z) = \frac{\vth(c_m x z^{-1} a_m^{-1}
)}{\vth(c_m z^{-1} )}
  \prod_{i<m} \vth(x/a_i)  \prod_{i>m} \vth(\hbar x/a_i) \,, \quad c_m
  = (-1)^n \hbar^{m-n/2} \,, 
\end{equation}
and set 
\begin{equation}
  \label{eq:29}
  \tgam_\mu = \Symm \prod_{i<j} \frac{\vth(\hbar
    x_i/x_j)}{\vth(x_i/x_j)}  \prod \bff_{\mu_i} (x_i, z \hbar^{2
    \rho_i}) \,, 
\end{equation}
where
\begin{enumerate}
\item[(1)] we symmetrize over the Weyl group $S(k)$ of $\bG$,
\item[(2)] the
collections
$$
1 \le \mu_1 < \mu_2 < \dots < \mu_k \le n
$$
correspond bijectively to points of $\bY^\bA$,
\item[(3)] $2\rho = (k-1,k-3,\dots,1-k)$  is the sum of positive roots of $\bG$.
\end{enumerate}
Note that first order poles along $x_i=x_j$ cancel upon
symmetrization, hence \eqref{eq:29}  is a regular function of all
variables.

\subsubsection{}
Consider a $1$-parameter subgroup $\sigma$ of the following form
$$
\sigma_l(q) = (1, \dots, 1, q, 1, \dots, 1)\in \bT_\bG \,, \quad
l=1,\dots, k  \,, 
$$
where $l$th coordinate is nontrivial. With respect to the quadratic
form \eqref{formT12}, we have
\begin{equation}
\sigma_l^\vee = x_l^n \prod a_i^{-1}\,, \quad \|\sigma_l\|^2 = n \,,
\quad 
\langle \sigma_l, \det T \rangle = n \,. \label{eq:43}
\end{equation}
The function \eqref{eq:29} satisfies
\begin{equation}
  \label{eq:42}
  \frac{\tgam_\mu( \sigma_l(q) x)}{\tgam_\mu(x)} =
  z q^{-n} \hbar^{-n/2} \left(\sigma_l^\vee\right)^{-1} \,, 
\end{equation}
which is a special case of \eqref{eq:14} in view of  \eqref{eq:43}. In
fact, each of the $k!$ terms in the formula \eqref{eq:29} satisfies
this difference equation.

\subsubsection{}

We now transform the integral \eqref{eq:37} into the vertex function
with descendants following the logic of the proof of Theorem \ref{t3}
in reverse.

The integral \eqref{t3} is 
the left-hand side of \eqref{vvv3} paired with a test function
$\rho$. As explained in Section \ref{s_pvv2_b}, the proof the
equality \eqref{vvv3} consist of two parts, the second of which
shows that it is enough to integrate over the maps $\bM_2 \subset
\bM_3$ that are stable at the generic point of $\Dd$.

In the language of Mellin-Barnes integrals, this second step, namely
the bound in \eqref{eq:28}, means that we
can deform the contour of integration to a contour of the form
\begin{equation}
\{|x_i|=1\} \rightsquigarrow \{|x_i|=q^{\lambda_i}\} \,, \label{eq:40}
\end{equation}
%
%
%
on which the integral picks up a factor of the form $\cO((\const
z)^{\sum \lambda_i})$. Here we choose the sign of $\lambda$
according to the stability condition, namely
$$
|z| \to 
\begin{cases}
  0 \,, & \textup{for $\cL_+$} \,, \\
  \infty \,, & \textup{for $\cL_-$} \,. 
\end{cases}
$$
See Section 3 and the Appendix in \cite{AFO} for a comprehensive
discussion of such contour deformations in the context of integral
formulas for vertex functions. A concrete description of the contour
deformation in the case at hand is given in Section \ref{s_poles}
below.

\subsubsection{}
The residues which we pick up in the process of this contour
deformation are interpreted as the contributions of $\bA \times
\Ct_\hbar \times \Ct_q$-fixed points in $\bM_2$.

Focusing on $\cL_+$-stable maps, these are characterized by $\cV$
being generated by sections $\xi^i A$ in diagram \eqref{eq135} at the generic point of
$\Dd$. Those fixed by the torus  $\bA \times
\Ct_\hbar \times \Ct_q$ are direct sums of the following blocks
\vspace{-1cm} 
\begin{equation}
  \label{diagAx1} \vspace{-1.3cm} 
  \begin{tikzcd}
    \cW_l = \cO_{\Dd} \otimes a_l   \arrow[rr,bend left, "A"] &&
    \cV_l=\bigoplus \cO_{\Dd}(d_i)\otimes a_l \hbar^{1-i} \arrow[
  out=-30,
  in=30,
  loop,
  distance=3.3cm, "\xi" '] 
  \end{tikzcd} \,, 
\end{equation} 
where $\cV_l$ has the form 
\begin{equation}
  \label{diagAx2} 
  \begin{tikzcd}
   \hspace{-1cm} \arrow[r,bend left, start anchor={[yshift=4ex,xshift=5ex]},"A"]  & \cO_{\Dd}(d_1)\otimes a_l   \arrow[r,bend left, "\xi"] &  \cO_{\Dd}(d_2)\otimes a_l
    \hbar^{-1}  \arrow[r,bend left, "\xi"] & \cO_{\Dd}(d_3)\otimes a_l
    \hbar^{-2}  \arrow[r,bend left, end
    anchor={[yshift=4ex,xshift=-6ex]}, "\xi"] & \hspace{-1cm} \dots 
  \end{tikzcd} \,, 
\end{equation} 
with
\begin{equation}
0 \le d_1 < d_2 < d_3 < \dots \,. \label{eq:46}
\end{equation}
In \eqref{diagAx2}, the maps are the natural inclusions
$$
\cO_\Dd(d_i) = t^{-d_i} \C[[t]] \hookrightarrow t^{-d_{i+1}}  \C[[t]] =
\cO_\Dd(d_{i+1}) \,. 
$$
The inequalities in \eqref{eq:46} are strict because $\xi$ has to be a
section of $\Hom(\cV,\cV) \otimes \cK_{\Dd}$, where
$\cK_{\Dd}=\cO_{\Dd}(-1)$.

\subsubsection{}\label{s_poles}

If $\cV$ is a direct sum of blocks of the form \eqref{diagAx1}  then at the corresponding
fixed point we have 
\begin{equation}
\{x_i\} = \bigcup_{l=1}^n \{ q^{d_{1,l}} a_l, q^{d_{2,l}} \hbar^{-1} a_l, q^{d_{3,l}}
\hbar^{-2} a_l, \dots \}  =  \bigcup_{l=1}^n \bigcup_{i=1}^{v_l} \{
q^{d_{i,l}} \hbar^{1-i} a_l\}\,. \label{pole}
\end{equation}
These can be directly matched to residues in the integral
\eqref{eq:37} as follows.

We write $\bPhi = \bPhi_a \bPhi_\xi$, where 
\begin{equation}
  \label{eq236}
   \bPhi_a = \prod_{i=1}^k \prod_{j=1}^n\frac{1}{\phi(a_j/x_i) \,
     \phi(\hbar x_i/a_j)}\,, \quad
   \bPhi_\xi = 
\prod_{i\ne j \le k} \frac{\phi^\circ(x_j/x_i)}{ \phi(q x_j/\hbar x_i)}
\,. 
\end{equation}
As explained e.g.\ in the Appendix in \cite{AFO}, the general prescription
is to deform the contour
in the direction of the gradient of the character
$$
{\det}^{-1} = \prod x_i^{-1}
$$
that defines the stability parameter $\cL_+$. This means the
deformation
$$
\{|x_i|=1\} \rightsquigarrow \{|x_i|=e^{-s}\}\,,  \quad \Re s >0\,, 
$$
until the contour meets a pole of the form
$$
x_j = q^{d_1} a_l \,, \quad d_1 \ge 0 \,, l \in \{1,\dots,n\} \,, 
$$
in the $\bPhi_a$ factor in \eqref{eq236}. Note that $\bPhi_\xi$ does
not pick up poles under this deformation.

\begin{figure}[!h]
  \centering
  \includegraphics[scale=0.5]{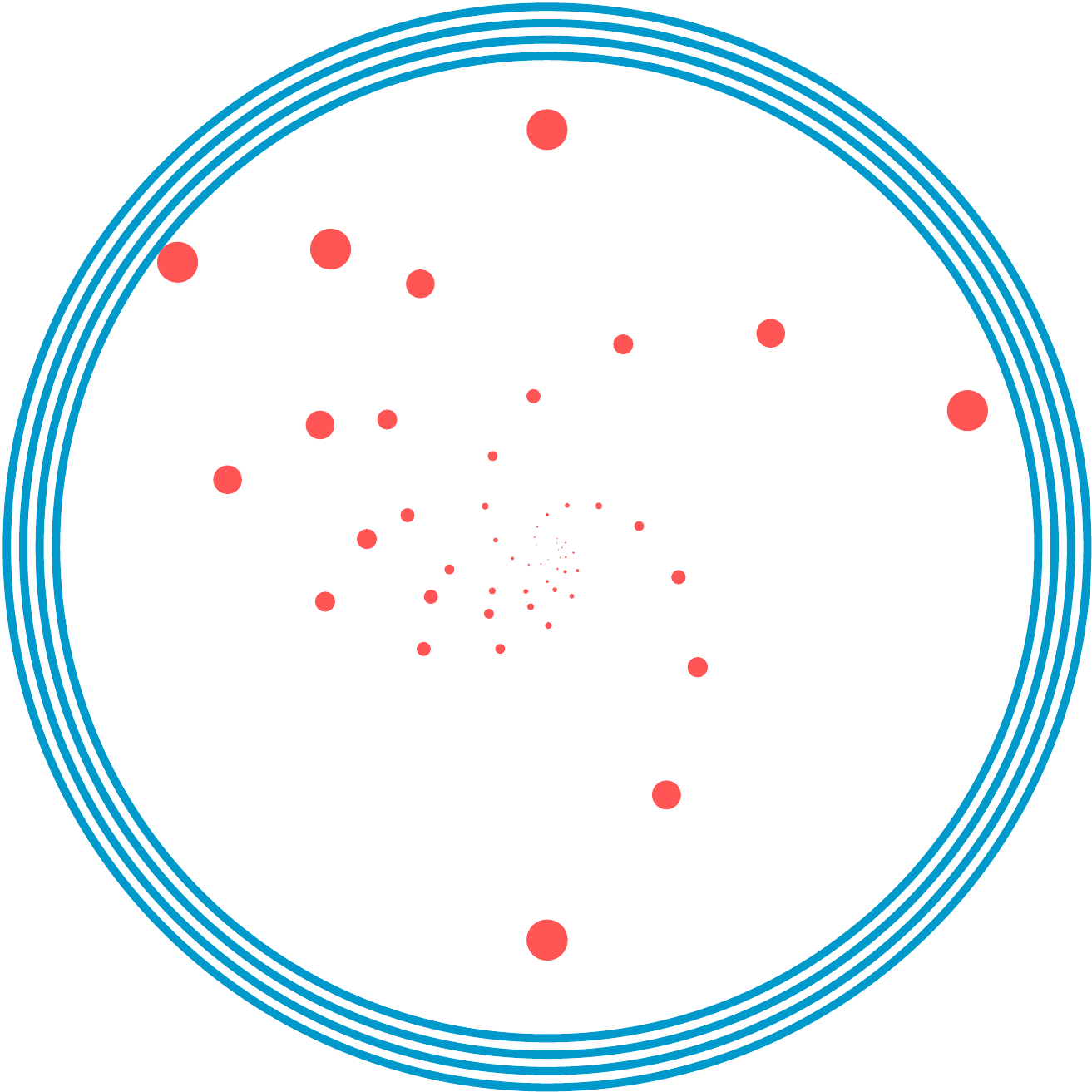}
  \caption{The original contours of integration $\{|x_i|=\const
    \approx 1\}$ and the poles of $\Phi_a$. The latter form geometric
    progressions of the form $q^d a_l$, where $|q|<|a_l|<1$.
    For $|z|\ll 1$, we deform the contours to $\{|x_i|=\const \ll
    1\}$.}
\label{countours1}
\end{figure}

\subsubsection{}

Once we specialize $x_j$ to $ q^{d_j} a_l$, we can move the product
$\prod_{i\ne j}  \phi(q x_j/\hbar x_i)$ from the denominator of
$\bPhi_\xi$ to the denominator of $\bPhi_a$, after which we similarly
deform all variables except $x_j$ in the direction of the gradient of
${\det}^{-1}$. Since the term $\phi^\circ(x_j/x_i)$ in the numerator
cancels the rest of the poles in the $\{q^d a_l\}$ series, this is equivalent to the
transformation
$$
\{a_1,\dots,a_l,\dots,a_n\} \mapsto \{a_1,\dots,q^{d_j+1} a_l/\hbar
,\dots,a_n\} \,. 
$$
Its iterations generate the poles \eqref{pole} \,.

\begin{figure}[!h]
  \centering
  \includegraphics[scale=0.64]{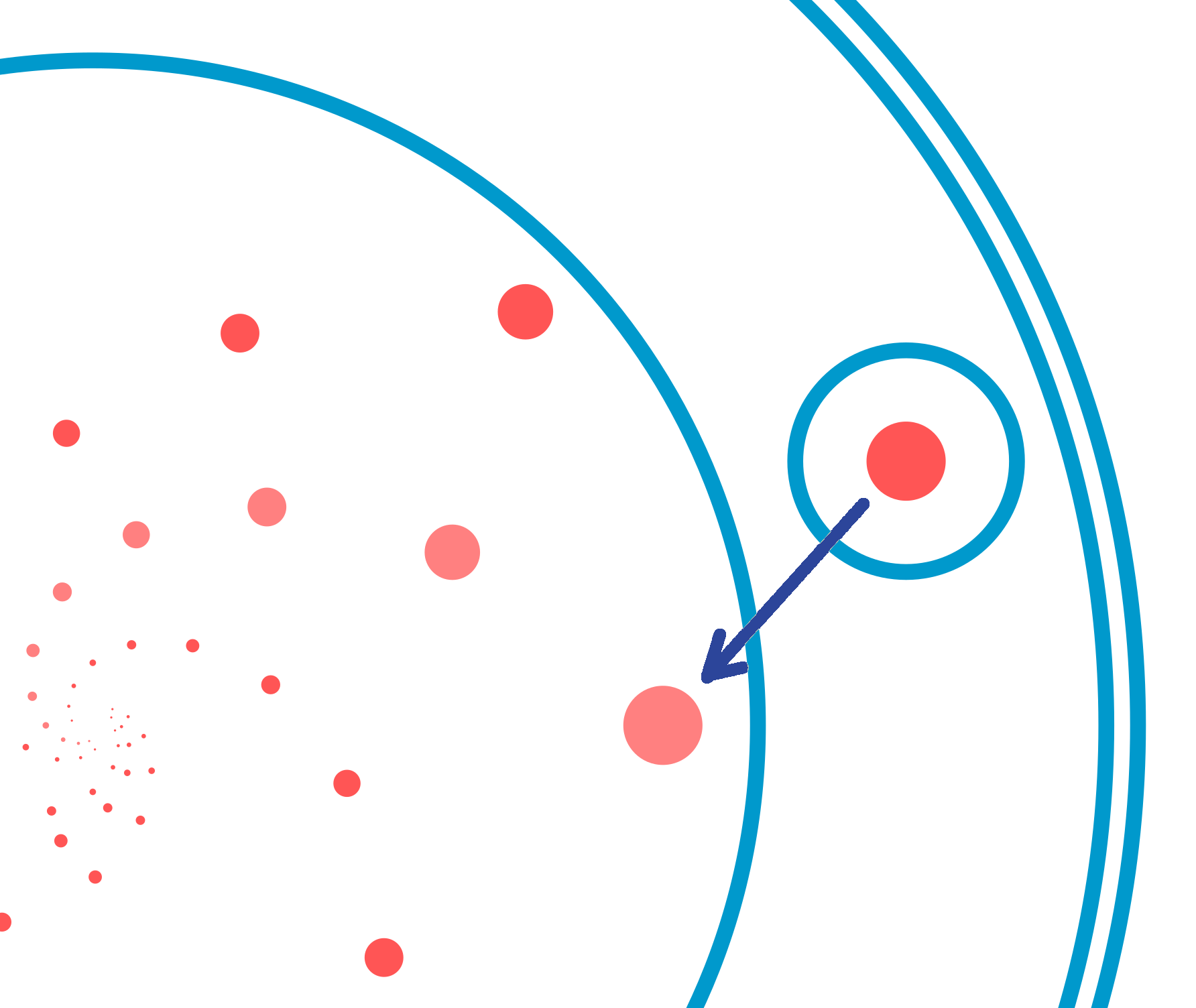}
  \caption{When pick up a residue at $x_j=q^{d_j} a_l$, the remaining
    poles in the $x_i \in \{q^{d_i} a_l\}$ series turn into poles of the
    form $x_i \in \{q^{d_i} a_l/\hbar\}$ with $d_i > d_j$.}
\label{fpolytope}
\end{figure}

\subsubsection{} 

To complete the equivalent of Proposition \ref{p_vv2} in our current
setting, we must show that poles \eqref{pole} that do not
correspond to maps to $\bY$, that is, poles such $v_l >1$ for
some $l$, do not contribute to the integral. This follows from the
following

\begin{Lemma}\label{l_wheel} 
  The functions $\tgam_\mu$ vanishes on the codimension  two locus 
  defined by 
  \begin{equation}
  \exists l,  \, \{a_l,\hbar^{-1} a_l\} \subset \{x_i\} \mod q^\Z \,.  \label{eq:47}
\end{equation}
\end{Lemma}

\noindent
The general abstract form of this Lemma is the content of argument in
Section \ref{s_wheel}.

Over the locus in $\Ell_{\bT_\bG \times \bA \times
  \Ct_\hbar}(\pt)$ that corresponds to \eqref{eq:47}, we have a fixed
point in $\bX$ that does not satisfy the moment map equation, namely
\begin{equation}
  \label{eq1343}
  \begin{tikzcd}
    W = a_l \oplus \dots   \arrow[rr, rounded corners, to path={
     -- ([yshift=2ex]\tikztostart.north) -| 
  ([yshift=3.3ex,xshift=-4.1ex]\tikztotarget.south)}] && V =
    a_l \oplus  \hbar^{-1} a_l \oplus \dots \arrow[ll,  rounded corners, to path={
     -- ([yshift=-2ex]\tikztostart.south) -| 
  ([yshift=-3.3ex,xshift=0ex]\tikztotarget.north)}]
  \end{tikzcd} \,, 
\end{equation} 
where we describe the vectors by their weights, the indicated maps are
nonzero equivariant maps, and all other maps are zero. Therefore any
class that is supported on $\mu^{-1}(0)\subset \bX$ has to vanish on
this locus.

Conditions of the kind \eqref{eq:47} that describe cohomology classes
supported on $\mu^{-1}(0)\subset \bX$ are popular in the literature
under the name \emph{wheel conditions}.  In the case at hand, the verification
of the wheel condition for $\tgam_\mu$ is elementary and goes as
follows.

\subsubsection{}

\begin{proof}[Proof of Lemma \ref{l_wheel}]
  We have
  \begin{equation}
  \tgam_\mu = \sum_{\tau \in S(k)} \prod_{\tau(i)<\tau(j)} \frac{\vth(\hbar
    x_i/x_j)}{\vth(x_i/x_j)}  \prod \bff_{\mu_{\tau(i)}} (x_i, z \hbar^{2
    \rho_{\tau(i)}}) \,. 
\label{eq:45}
\end{equation}
  By symmetry we may suppose that
  $$
  x_1 = a_l, \quad x_2 = a_l/\hbar \,, 
  $$
  and note that the  denominator in \eqref{eq:45} does not vanish upon
  this specialization. 
  
  We have
  $$
  \bff_{\mu_{\tau(1)}} (a_l, \dots )   \, \bff_{\mu_{\tau(2)}}
  (a_l/\hbar, \dots ) \ne 0 \quad \Rightarrow \quad 
  \mu_{\tau(1)} \le l \le \mu_{\tau(2)} \,, 
  $$
  and since $\mu_1 < \mu_2 < \dots$, this means that terms with
  $\tau(1) > \tau(2)$ vanish in \eqref{eq:45}. On the other hand, 
  $$
  \tau(1) < \tau(2)  \Rightarrow \prod_{\tau(i)<\tau(j)} \vth(\hbar
  x_i/x_j) = 0 \,, 
  $$
  which concludes the proof. 
\end{proof}

\subsubsection{}
It remains to consider the residue of the integral at the pole of the
form
$$
x = \{ q^{d_i} a_{\eta_i} \}
$$
where $\eta_i \ne \eta_j$. Since the integrand satisfies a scalar
$q$-difference equation, this reduces to the case $d_i=0$, that is,
to computation with constant maps.

The analysis of the scalar
$q$-difference equation is carried out in Sections \ref{s_tgam}
--- \ref{s_pvv1_c}, reproducing, in particular, formula
\eqref{eq:42} above.  This concludes direct verification of the
Proposition \ref{p_MB} in our running example.


\begin{bibdiv}
	\begin{biblist}


         


\bib{AFO}{article}{
   author={Aganagic, M.},
   author={Frenkel, E.},
   author={Okounkov, A.},
   title={Quantum $q$-Langlands correspondence},
   journal={Trans. Moscow Math. Soc.},
   volume={79},
   date={2018},
   pages={1--83},
   }

	
       

\bibitem{ese}
  M.~Aganagic and A.~Okounkov,
  \emph{Elliptic stable envelopes}, JAMS,
  \texttt{arXiv:1604.00423}.

\bib{Bethe}{article}{
   author={Aganagic, Mina},
   author={Okounkov, Andrei},
   title={Quasimap counts and Bethe eigenfunctions},
   journal={Mosc. Math. J.},
   volume={17},
   date={2017},
   number={4},
   pages={565--600},
}
  

\bibitem{AO2}
  M.~Aganagic and A.~Okounkov,
  \emph{Duality interfaces in 3-dimensional theories},
  talks at StringMath2019, available from
  \url{https://www.stringmath2019.se/scientific-talks-2/}.


  


  




\bibitem{BezOk}
  R.~Bezrukavnikov and A.~Okounkov,
  \emph{Monodromy and derived equivalences},
  in preparation. 



  




\bib{Bogo}{article}{
   author={Bogomolov, F. A.},
   title={Holomorphic tensors and vector bundles on projective manifolds},
   journal={Izv. Akad. Nauk SSSR Ser. Mat.},
   volume={42},
   date={1978},
   number={6},
}





\bib{mixp}{article}{
   author={Chang, Huai-Liang},
   author={Li, Jun},
   author={Li, Wei-Ping},
   author={Liu, Chiu-Chu Melissa},
   title={On the mathematics and physics of mixed spin P-fields},
   conference={
      title={String-Math 2015},
   },
   book={
      series={Proc. Sympos. Pure Math.},
      volume={96},
      publisher={Amer. Math. Soc., Providence, RI},
   },
   date={2017},
   pages={47--73},
}


\bibitem{matrix} 
  Ionut Ciocan-Fontanine, David Favero, J\'er\'emy Gu\'er\'e, Bumsig
  Kim, Mark Shoemaker,
  \emph{Fundamental Factorization of a GLSM, Part I: Construction},
  \texttt{arXiv:1802.05247}. 


\bib{CFKM}{article}{
   author={Ciocan-Fontanine, Ionu\c{t}},
   author={Kim, Bumsig},
   author={Maulik, Davesh},
   title={Stable quasimaps to GIT quotients},
   journal={J. Geom. Phys.},
   volume={75},
   date={2014},
   pages={17--47},
}



\bib{Gaz}{article}{
   author={Di Vizio, L.},
   author={Ramis, J.-P.},
   author={Sauloy, J.},
   author={Zhang, C.},
   title={\'Equations aux $q$-diff\'erences},
   language={French},
   journal={Gaz. Math.},
   number={96},
   date={2003},
   pages={20--49},
}

\bib{Hori_Beij}{article}{
   author={Eager, Richard},
   author={Hori, Kentaro},
   author={Knapp, Johanna},
   author={Romo, Mauricio},
   title={Beijing lectures on the grade restriction rule},
   journal={Chin. Ann. Math. Ser. B},
   volume={38},
   date={2017},
   number={4},
   pages={901--912},
}


\bib{EFK}{book}{
   author={Etingof, Pavel I.},
   author={Frenkel, Igor B.},
   author={Kirillov, Alexander A., Jr.},
   title={Lectures on representation theory and Knizhnik-Zamolodchikov
   equations},
   series={Mathematical Surveys and Monographs},
   volume={58},
   publisher={American Mathematical Society, Providence, RI},
   date={1998},
   pages={xiv+198},
}







\bib{FrenResh}{article}{
   author={Frenkel, I. B.},
   author={Reshetikhin, N. Yu.},
   title={Quantum affine algebras and holonomic difference equations},
   journal={Comm. Math. Phys.},
   volume={146},
   date={1992},
   number={1},
   pages={1--60},
}



  
  
  















\bib{DHL}{article}{
   author={Halpern-Leistner, Daniel},
   title={The derived category of a GIT quotient},
   journal={J. Amer. Math. Soc.},
   volume={28},
   date={2015},
   number={3},
   pages={871--912},
}

\bibitem{DHLloc} D.~Halpern-Leistner,
  \emph{A categorification of the Atiyah-Bott localization formula},
  available from \texttt{math.cornell.edu/~danielhl}. 

\bibitem{DHLMO}
  D.~Halpern-Leistner, D.~Maulik, A.~Okounkov,
  \emph{Caterogorical stable envelopes and magic windows}, 
  in preraration.

  \bib{HLS}{article}{
   author={Halpern-Leistner, Daniel},
   author={Sam, Steven V.},
   title={Combinatorial constructions of derived equivalences},
   journal={J. Amer. Math. Soc.},
   volume={33},
   date={2020},
   number={3},
   pages={735--773},
}


  

\bib{Hess}{article}{
   author={Hesselink, Wim H.},
   title={Uniform instability in reductive groups},
   journal={J. Reine Angew. Math.},
   volume={303(304)},
   date={1978},
   pages={74--96},
}

\bib{Hori_Tong}{article}{
   author={Hori, Kentaro},
   author={Tong, David},
   title={Aspects of non-abelian gauge dynamics in two-dimensional $\scr
   N=(2,2)$ theories},
   journal={J. High Energy Phys.},
   date={2007},
   number={5},
   pages={079, 41},
}



\bib{Iri}{article}{
   author={Iritani, Hiroshi},
   title={An integral structure in quantum cohomology and mirror symmetry
   for toric orbifolds},
   journal={Adv. Math.},
   volume={222},
   date={2009},
   number={3},
   pages={1016--1079},
}






\bibitem{cosec}
  Young-Hoon Kiem, Jun Li,
  \emph{Localizing virtual structure sheaves by cosections},
 \texttt{arXiv:1705.09458}. 

\bib{Kempf}{article}{
   author={Kempf, George R.},
   title={Instability in invariant theory},
   journal={Ann. of Math. (2)},
   volume={108},
   date={1978},
   number={2},
   pages={299--316},
}






\bib{Matsuo}{article}{
   author={Matsuo, Atsushi},
   title={Jackson integrals of Jordan-Pochhammer type and quantum
   Knizhnik-Zamolodchikov equations},
   journal={Comm. Math. Phys.},
   volume={151},
   date={1993},
   number={2},
   pages={263--273},
}



\bib{MO1}{article}{
   author={Maulik, Davesh},
   author={Okounkov, Andrei},
   title={Quantum groups and quantum cohomology},
   language={English, with English and French summaries},
   journal={Ast\'{e}risque},
   number={408},
   date={2019},
   pages={ix+209},
}

\bib{Higgs}{article}{
   author={Moore, Gregory},
   author={Nekrasov, Nikita},
   author={Shatashvili, Samson},
   title={Integrating over Higgs branches},
   journal={Comm. Math. Phys.},
   volume={209},
   date={2000},
   number={1},
   pages={97--121},
}



\bibitem{Mus}
M.~Musta\c t\u a, Spaces of arcs in birational geometry, available
at http://www.math.lsa.umich.edu



\bib{McGN}{article}{
   author={McGerty, Kevin},
   author={Nevins, Thomas},
   title={Kirwan surjectivity for quiver varieties},
   journal={Invent. Math.},
   volume={212},
   date={2018},
   number={1},
   pages={161--187}
}

\bib{Merk}{article}{
   author={Merkurjev, Alexander S.},
   title={Equivariant $K$-theory},
   conference={
      title={Handbook of $K$-theory. Vol. 1, 2},
   },
   book={
      publisher={Springer, Berlin},
   },
   date={2005},
   pages={925--954}
}





\bib{Nak1}{article}{
   author={Nakajima, Hiraku},
   title={Instantons on ALE spaces, quiver varieties, and Kac-Moody
   algebras},
   journal={Duke Math. J.},
   volume={76},
   date={1994},
   number={2},
   pages={365--416},
}





\bib{Nekr_Instantons}{article}{
   author={Nekrasov, Nikita A.},
   title={Seiberg-Witten prepotential from instanton counting},
   journal={Adv. Theor. Math. Phys.},
   volume={7},
   date={2003},
   number={5},
   pages={831--864},
}

\bib{NO}{article}{
   author={Nekrasov, Nikita},
   author={Okounkov, Andrei},
   title={Membranes and sheaves},
   journal={Algebr. Geom.},
   volume={3},
   date={2016},
   number={3},
   pages={320--369},
}


\bib{NS}{article}{
   author={Nekrasov, Nikita A.},
   author={Shatashvili, Samson L.},
   title={Supersymmetric vacua and Bethe ansatz},
   journal={Nuclear Phys. B Proc. Suppl.},
   volume={192/193},
   date={2009},
   pages={91--112},
}

\bib{Ness}{article}{
   author={Ness, Linda},
   title={A stratification of the null cone via the moment map},
   note={With an appendix by David Mumford},
   journal={Amer. J. Math.},
   volume={106},
   date={1984},
   number={6},
   pages={1281--1329},
}


\bib{pcmi}{article}{
   author={Okounkov, Andrei},
   title={Lectures on K-theoretic computations in enumerative geometry},
   conference={
      title={Geometry of moduli spaces and representation theory},
   },
   book={
      series={IAS/Park City Math. Ser.},
      volume={24},
      publisher={Amer. Math. Soc., Providence, RI},
   },
   date={2017},
   pages={251--380},
}



\bib{slc}{article}{
   author={Okounkov, Andrei},
   title={Enumerative geometry and geometric representation theory},
   conference={
      title={Algebraic geometry: Salt Lake City 2015},
   },
   book={
      series={Proc. Sympos. Pure Math.},
      volume={97},
      publisher={Amer. Math. Soc., Providence, RI},
   },
   date={2018},
   pages={419--457},
 }

 \bib{Takagi}{article}{
   author={Okounkov, Andrei},
   title={Takagi lectures on Donaldson-Thomas theory},
   journal={Jpn. J. Math.},
   volume={14},
   date={2019},
   number={1},
   pages={67--133},
}

 
\bib{icm}{article}{
   author={Okounkov, Andrei},
   title={On the crossroads of enumerative geometry and geometric
   representation theory},
   conference={
      title={Proceedings of the International Congress of
      Mathematicians---Rio de Janeiro 2018. Vol. I. Plenary lectures},
   },
   book={
      publisher={World Sci. Publ., Hackensack, NJ},
   },
   date={2018},
   pages={839--867},
}

\bibitem{part1}
\bysame, 
\emph{Inductive construction of stable envelopes and applications,
  I. Actions of tori. Elliptic cohomology and K-theory},
\texttt{arXiv:2007.09094}



\bibitem{OS}
A.~Okounkov and A.~Smirnov, 
\emph{Quantum difference equations for Nakajima varieties},
\texttt{arXiv:1602.09007}. 


  


\bib{Reshet_int}{article}{
   author={Reshetikhin, N.},
   title={Jackson-type integrals, Bethe vectors, and solutions to a
   difference analog of the Knizhnik-Zamolodchikov system},
   journal={Lett. Math. Phys.},
   volume={26},
   date={1992},
   number={3},
   pages={153--165},
   issn={0377-9017},
   review={\MR{1199739}},
   doi={10.1007/BF00420749},
}


\bib{Rous}{article}{
   author={Rousseau, Guy},
   title={Immeubles sph\'{e}riques et th\'{e}orie des invariants},
   language={French, with English summary},
   journal={C. R. Acad. Sci. Paris S\'{e}r. A-B},
   volume={286},
   date={1978},
   number={5},
   pages={A247--A250},
}






 \bibitem{Smir_desc}
A.~Smirnov, 
\emph{Rationality of capped descendent vertex in K-theory},
\texttt{arXiv:1612.01048}. 









\bib{Varchenko1}{article}{
   author={Varchenko, Alexander},
   title={Quantized Knizhnik-Zamolodchikov equations, quantum Yang-Baxter
   equation, and difference equations for $q$-hypergeometric functions},
   journal={Comm. Math. Phys.},
   volume={162},
   date={1994},
   number={3},
   pages={499--528},
}


\bib{VinPop}{article}{
   author={Vinberg, E. B.},
   author={Popov, V. L.},
   title={Invariant theory},
   book={
      series={Itogi Nauki i Tekhniki},
      publisher={Akad. Nauk SSSR, Vsesoyuz. Inst. Nauchn. i Tekhn. Inform.,
   Moscow},
   },
   date={1989},
   pages={137--314, 315},
}

  

	\end{biblist}
\end{bibdiv}
\end{document}